\documentclass[a4paper,12pt,reqno]{amsart}
\usepackage{amsfonts}
\usepackage{amsmath}
\usepackage{amssymb}
\usepackage[a4paper]{geometry}
\usepackage{mathrsfs}
\usepackage[colorlinks]{hyperref}
\renewcommand\eqref[1]{(\ref{#1})} 
\numberwithin{equation}{section}
\theoremstyle{plain}
\newtheorem{thm}{Theorem}[section]

\newtheorem{cor}[thm]{Corollary}

\theoremstyle{definition}

\newtheorem{rem}[thm]{Remark}

\renewcommand{\wp}{\mathfrak S}
\newcommand{\Rn}{\mathbb R^{n}}
\def\R{\mathcal R}
\def\L{\mathcal L}
\def\G{{\mathbb G}}

\begin{document}

   \title[Factorizations on stratified groups]
   {Factorizations and Hardy-Rellich inequalities on stratified groups}

\author[M. Ruzhansky]{Michael Ruzhansky}
\address{
  Michael Ruzhansky:
  \endgraf
  Department of Mathematics
  \endgraf
  Imperial College London
  \endgraf
  180 Queen's Gate, London SW7 2AZ
  \endgraf
  United Kingdom
  \endgraf
  {\it E-mail address} {\rm m.ruzhansky@imperial.ac.uk}
  }

\author[N. Yessirkegenov]{Nurgissa Yessirkegenov}
\address{
  Nurgissa Yessirkegenov:
  \endgraf
  Institute of Mathematics and Mathematical Modelling
  \endgraf
  125 Pushkin str.
  \endgraf
  050010 Almaty
  \endgraf
  Kazakhstan
  \endgraf
  and
  \endgraf
  Department of Mathematics
  \endgraf
  Imperial College London
  \endgraf
  180 Queen's Gate, London SW7 2AZ
  \endgraf
  United Kingdom
  \endgraf
  {\it E-mail address} {\rm n.yessirkegenov15@imperial.ac.uk}
  }

\thanks{The authors were supported in parts by the EPSRC
 grant EP/K039407/1 and by the Leverhulme Grant RPG-2014-02, as well as by the MESRK grant 0825/GF4. No new data was collected or
generated during the course of research.}


     \begin{abstract}
     In this paper, we obtain Hardy, Hardy-Rellich and refined Hardy inequalities on general stratified groups and weighted Hardy inequalities on general homogeneous groups using the factorization method of differential operators, inspired by the recent work of
     Gesztesy and Littlejohn \cite{GL17}.
We note that some of the obtained inequalities are new also in the usual Euclidean setting.
We also obtain analogues of Gesztesy and Littlejohn's 2-parameter version of the Rellich inequality on stratified groups and on the Heisenberg group, and a new two-parameter estimate on $\Rn$ which can be regarded as a counterpart to the Gesztesy and Littlejohn's estimate.
     \end{abstract}
     \maketitle

\tableofcontents

\section{Introduction}
\label{SEC:intro}
Consider the Hardy inequality
\begin{equation}\label{Hardy}
\int_{\Rn}|(\nabla f)(x)|^{2}dx\geq \left(\frac{n-2}{2}\right)^{2}\int_{\Rn}\frac{|f(x)|^{2}}{|x|^{2}}dx,\;\;n\geq3,
\end{equation}
for all functions $f\in C_{0}^{\infty}(\Rn\backslash\{0\})$, where $\nabla$ is the standard gradient in $\Rn$. The purpose of this paper is to obtain Hardy, weighted Hardy, improved Hardy and Hardy-Rellich inequalities on stratified, on Heisenberg and on homogeneous groups by factorizing differential expressions.

Therefore, let us first recall known results in this direction. In the recent paper \cite{GL17}, inspiring our work, Gesztesy and Littlejohn used the nonnegativity of $T^{+}_{\alpha,\beta}T_{\alpha,\beta}$ on $C_{0}^{\infty}(\Rn\backslash\{0\})$, for two-parameter differential expressions $$T_{\alpha,\beta}:=-\triangle+\alpha|x|^{-2}x\cdot\nabla+\beta|x|^{-2}$$ and its formal adjoint
$$
T^{+}_{\alpha,\beta}:=-\triangle-\alpha|x|^{-2}x\cdot\nabla+(\beta-\alpha(n-2))|x|^{-2},
$$
for $\alpha,\beta\in\mathbb{R}$, $x\in\Rn\backslash\{0\}$ and $n\geq2$.
As a result, they obtained the following inequality for all $f\in C_{0}^{\infty}(\Rn\backslash\{0\})$:
\begin{equation}\label{Gesztesy_ineq}
\begin{split}
\int_{\Rn}|(\triangle f)(x)|^{2}dx\geq&((n-4)\alpha-2\beta)\int_{\Rn}|x|^{-2}|(\nabla f)(x)|^{2}dx\\
&-\alpha(\alpha-4)\int_{\Rn}|x|^{-4}|x\cdot(\nabla f)(x)|^{2}dx\\
&+\beta((n-4)(\alpha-2)-\beta)\int_{\Rn}|x|^{-4}|f(x)|^{2}dx,
\end{split}
\end{equation}
which implies classical Rellich and Hardy-Rellich-type inequalities after choosing special values of parameters $\alpha$ and $\beta$.

We refer to \cite{GL17} for a thorough discussion of the factorization method, its history and different features.
We also refer to \cite{GP80} for obtaining the Hardy inequality and to \cite{G84} for logarithmic refinements by this factorization method.

We note that among other things, in Corollary \ref{cor_Rn2}
we show the counterpart of the Gesztesy and Littlejohn estimate \eqref{Gesztesy_ineq} which follows from the nonnegativity of the operator $T_{\alpha,\beta}T^{+}_{\alpha,\beta}$, namely,
 for $\alpha,\beta\in \mathbb{R}$ and $f\in C_{0}^{\infty}(\Rn\backslash\{0\})$ with $n\geq2$, we have
\begin{equation}\label{heisen_Rellich11_Rn0}
\begin{split}
 \int_{\Rn}|(\triangle f)(x)|^{2}dx
&\geq (n\alpha-2\beta)\int_{\Rn}|x|^{-2}|(\nabla f)(x)|^{2}dx
\\ & -\alpha(\alpha+4)\int_{\Rn}|x|^{-4}|x\cdot(\nabla f)(x)|^{2}dx\\
&+(2(n-4)(\alpha(n-2)-\beta)-2\alpha^{2}(n-2)+\alpha\beta n-\beta^{2})\times
\\ & \quad \times\int_{\Rn}|x|^{-4}|f(x)|^{2}dx.
\end{split}
\end{equation}

For the convenience of the reader let us now shortly recapture the main results of this paper, first for general homogeneous groups, then for stratified groups, and finally for the Heisenberg group, where more precise expressions are possible due to the possibility of using explicit formulae for its commutators.

The following result is new already in the usual setting of $\Rn$, however, we may also formulate it in the general setting of homogeneous groups in the sense of Folland and Stein \cite{FS-Hardy}. We recall the details of such construction in Section \ref{SEC:prelim}.
\begin{itemize}
\item Let $\mathbb{G}$ be a homogeneous group
of homogeneous dimension $Q\geq 3$ and let $|\cdot|$ be an arbitrary homogeneous quasi-norm.
Let $\mathcal{R}:=\frac{d}{d|x|}$ be the radial derivative.
Let $\phi,\psi\in L_{loc}^{2}(\mathbb{G}\backslash\{0\})$ be any real-valued functions such that $\R \phi,\R \psi\in L_{loc}^{2}(\mathbb{G}\backslash\{0\})$. Let $\alpha\in\mathbb{R}$. Then for all complex-valued functions $f\in C_{0}^{\infty}(\mathbb{G}\backslash\{0\})$ we have
$$\int_{\G}(\phi(x))^{2}|\R f(x)|^{2}dx$$
$$\geq \alpha \int_{\G}\left(\phi(x)\R\psi(x)+\psi(x) \R \phi(x)+(Q-1)\frac{\phi(x)\psi(x)}{|x|}\right)|f(x)|^{2}dx$$
$$
-\alpha^{2}\int_{\G}(\psi(x))^{2}|f(x)|^{2}dx,
$$
and also
$$\int_{\G}(\phi(x))^{2}|\R f(x)|^{2}dx$$
$$\geq \alpha \int_{\G}\left(\psi(x) \R \phi(x)-\phi(x)\R \psi(x)+(Q-1)\frac{\phi(x)\psi(x)}{|x|}\right)|f(x)|^{2}dx$$
$$
-\alpha^{2}\int_{\G}(\psi(x))^{2}|f(x)|^{2}dx
$$
$$+\int_{\G}((Q-1)\frac{\phi(x)\R \phi(x)}{|x|}-(Q-1)\frac{(\phi(x))^{2}}{|x|^{2}}+\phi(x)\R^{2}\phi(x))|f(x)|^{2}dx.
$$
A number of consequences of these estimates for different choices of functions $\phi$ and $\psi$ are given in Section \ref{SEC:Fac_hom}. In particular, the above estimates are new already in the usual Euclidean setting, in which case we can take $\G=\Rn$, $Q=n$, $|x|$ is the usual Euclidean norm of $x$, and
$$\mathcal{R}f(x)=\frac{x}{|x|}\cdot\nabla f(x),\quad x\in\Rn,$$ is the usual radial derivative.
For general homogeneous groups the radial derivative operator $\mathcal{R}$ is explained in detail in formulae \eqref{dfdr0}-\eqref{dfdr}.
We can also refer to \cite{RSY17} for another type of weighted Hardy inequalities, the so-called Caffarelli-Kohn-Nirenberg inequalities on homogeneous groups, and to references therein.
\end{itemize}

Let us now list the results on a stratified group $\G$ of homogeneous dimension $Q\geq 3$. Let $X_{1},...,X_{N}$ be left-invariant vector fields giving the first stratum of the Lie algebra of $\G$, $\nabla_{H}=(X_{1},...,X_{N})$ the horizontal gradient, $(x',x^{(2)},\ldots,x^{(r)})\in \mathbb{R}^{N}\times\ldots\times\mathbb{R}^{N_{r}}$ with $N+N_{2}+\ldots+N_{r}=n$, $|\cdot|$ the Euclidean norm on $\mathbb{R}^{N}$, $r$ being the step of $\G$, and $\L$ the sub-Laplacian operator
$$\L=\sum_{k=1}^{N}X_{k}^{2}.$$
\begin{itemize}
\item Let $\mathbb{G}$ be a stratified group with $N\geq2$ being the dimension of the first stratum, and let $\alpha, \beta\in \mathbb{R}$. Then for all complex-valued functions $f\in C_{0}^{\infty}(\G\backslash\{x'=0\})$ we have
\begin{equation}\label{EQ:strat_ab}
\begin{split}
& \|\mathcal{L} f\|^{2}_{L^{2}(\G)} \\
&\geq (\alpha(N-2)-2\beta)\left\|\frac{\nabla_{H} f}{|x'|}\right\|^{2}_{L^{2}(\G)}
-\alpha^{2}\left\|\frac{x'\cdot\nabla_{H} f}{|x'|^{2}}\right\|^{2}_{L^{2}(\G)}\\
&+(\alpha(N-4)(N-2)-\alpha^{2}(N-2)+2\beta(4-N)-\beta^{2}+\alpha\beta(N-2))
\left\|\frac{f}{|x'|^{2}}\right\|^{2}_{L^{2}(\G)}.
\end{split}
\end{equation}
\item Let $\mathbb{G}$ be a stratified group with $N\geq3$ being the dimension of the first stratum. Then for all complex-valued functions $f\in C_{0}^{\infty}(\mathbb{G}\backslash\{x'=0\})$ we have the refined Hardy inequality
\begin{equation}\label{EQ:istrt-r}
 \left\|\frac{x'\cdot\nabla_{H}f}{|x'|}\right\|_{L^{2}(\G)}\geq \frac{N-2}{2}\left\|\frac{f}{|x'|}\right\|_{L^{2}(\G)},
\end{equation}
where the constant $\frac{N-2}{2}$ is sharp.
\end{itemize}
In particular, by the Cauchy-Schwartz inequality \eqref{EQ:istrt-r} implies the Hardy inequality
\begin{equation}\label{EQ:istrt-r2}
 \left\|\nabla_{H}f\right\|_{L^{2}(\G)}\geq \frac{N-2}{2}\left\|\frac{f}{|x'|}\right\|_{L^{2}(\G)},
\end{equation}
where the constant $\frac{N-2}{2}$ is also sharp. In fact, the sharpness of this constant was the nature of the Badiale-Tarantello conjecture \cite[Remark 2.3]{BadTar:ARMA-2002}. This conjecture was solved in \cite{Secchi-Smets-Willen} and further improvements were obtained in \cite{RS17a}.
In Theorem \ref{Hardy_strat_thm} we obtain yet another proof of \eqref{EQ:istrt-r2} showing that the factorization method gives an elementary half-page proof for it.

On stratified group, we note that the weighted Hardy-Rellich type inequalities, namely for $\left\|\frac{\L f}{|x'|^{a}}\right\|^{2}_{L^{2}(\G)}\left\|\frac{\nabla_{H} f}{|x'|^{b}}\right\|^{2}_{L^{2}(\G)}$ with $a,b\in\mathbb{R}$ instead of $\|\mathcal{L} f\|^{2}_{L^{2}(\G)}$ in \eqref{EQ:strat_ab}, were obtained in \cite{SS17}.

If we use the $\L$-gauge in \eqref{EQ:istrt-r2} instead of the norm $|x'|$, different versions of the Hardy inequality have been actively investigated on stratified groups, we refer to
\cite{DGP-Hardy-potanal}, \cite{GL}, \cite{GolKom}, \cite{Grillo:Hardy-Rellich-PA-2003}, \cite{Jin-Shen:Hardy-Rellich-AM-2011}, \cite{Kombe-Ozaydin:Hardy-Rellich-mfds}, \cite{Kogoj-Sonner:Hardy-lambda-CVEE-2016}, \cite{Lian:Rellich}, \cite{NZW-Hardy-p} for a partial list of references for this and related problems. For Hardy and Rellich inequalities for H\"ormander's sums of squares of vector fields we refer to \cite{RS17c} and references therein.

We finally state the analogue of the inequality \eqref{Gesztesy_ineq} on the Heisenberg group.
We note that compared to \eqref{Gesztesy_ineq}, the analogous expression of $T^{+}_{\alpha,\beta}T_{\alpha,\beta}$ contains additional terms involving commutators of invariant vector fields -- such terms do not appear in the Abelian Euclidean setting. In the case of general stratified groups, to get inequality \eqref{EQ:strat_ab}, we cancel these commutators by considering the non-negative expression
$T^{+}_{\alpha,\beta}T_{\alpha,\beta}+T_{\alpha,\beta}T^{+}_{\alpha,\beta}$ instead. However, in the case of the Heisenberg group, using explicit knowledge of left-invariant vector fields, we can calculate the inequalities corresponding to both $T^{+}_{\alpha,\beta}T_{\alpha,\beta}$ and $T_{\alpha,\beta}T^{+}_{\alpha,\beta}$, which will be given in \eqref{EQ:Heis1} and \eqref{EQ:Heis11}, respectively.

We briefly recall that the Heisenberg group $\mathbb{H}_{n}$ is a manifold $\mathbb{R}^{2n+1}$ with the group law given by
$$({x}^{(1)},{y}^{(1)},t^{(1)})({x}^{(2)},{y}^{(2)},t^{(2)}):=({x}^{(1)}+{x}^{(2)},{y}^{(1)}+{y}^{(2)},{t}^{(1)}+{t}^{(2)}
+\frac{1}{2}({x}^{(1)}{y}^{(2)}-{x}^{(2)}{y}^{(1)})),$$
for $({x}^{(1)},{y}^{(1)},t^{(1)}), ({x}^{(2)},{y}^{(2)},t^{(2)})\in \mathbb{R}^{n}\times\mathbb{R}^{n}\times\mathbb{R}\sim\mathbb{H}_{n}$, where we denote by ${x}^{(1)}{y}^{(2)}$ and ${x}^{(2)}{y}^{(1)}$ their usual scalar products on $\mathbb{R}^{n}$.
The canonical basis of its Lie algebra $\mathfrak{h}_{n}$ is given by the left-invariant vector fields
\begin{equation}\label{EQ:Heis-nots}
X_{j}=\partial_{x_{j}}-\frac{y_{j}}{2}\partial_{t},\;\;Y_{j}=\partial_{y_{j}}+
\frac{x_{j}}{2}\partial_{t},\;\;j=1,\ldots,n,\;\;T=\partial_{t}.
\end{equation}
The canonical commutation relations of the basis $\{X_{j},Y_{j},T\}$ for $j=1,...,n$ are given by
$$[X_{j},Y_{j}]=T,\quad j=1,\ldots,n,$$
with all the other commutators being zero.
We also note that the Heisenberg Lie algebra $\mathfrak{h}_{n}$ is stratified via $\mathfrak{h}_{n}=V_{1}\oplus V_{2}$, where $V_{1}$ is linearly spanned by the $X_{j}$'s and $Y_{j}$'s, and $V_{2}=\mathbb{R}T$. Therefore, the natural dilations on $\mathfrak{h}_{n}$ are given by
$$D_{r}(X_{j})=rX_{j},\;\;D_{r}(Y_{j})=rY_{j},\;\;D_{r}(T)=r^{2}T,$$
and on $\mathbb{H}_{n}$ by
$$D_{r}(x,y,t)=r(x,y,t)=(rx,ry,r^{2}t),\;\;(x,y,t)\in\mathbb{H}_{n},\;r>0.$$
Consequently, $Q=2n+2$ is the homogeneous degree of the Lebesgue measure $dxdydt$ and the homogeneous dimension of the Heisenberg group $\mathbb{H}_{n}$ as well. The sub-Laplacian is given by
$$\L:=\sum_{j=1}^{n}(X^{2}_{j}+Y^{2}_{j})=\sum_{j=1}^{n}\left(\partial_{x_{j}}-\frac{y_{j}}{2}\partial_{t}\right)^{2}+\left(\partial_{y_{j}}+
\frac{x_{j}}{2}\partial_{t}\right)^{2}.$$
To simplify and unify the following formulations, we denote by $\widetilde{x}'$ the variables of the first stratum:
\begin{equation}\label{heisen_Xi}
\widetilde{x}'_{j}=
\begin{cases} x_{j}, 1\leq j\leq n;\\
y_{j-n}, n+1\leq j\leq 2n,   \end{cases}
\end{equation}
with its dimension being $N=2n$, the norm
\begin{equation}\label{heisen_normi}
|\widetilde{x}'|=\sum_{j=1}^{N}\sqrt{({\widetilde{x}'}_{j})^{2}}=\sqrt{\sum_{j=1}^{n}(x^{2}_{j}+y^{2}_{j})},
\end{equation}
and the horizontal gradient
\begin{equation}\label{heisen_gradi}
\nabla_{H}:=({X}_{1}, \ldots, {X}_{n}, {Y}_{1}, \ldots, {Y}_{n}).
\end{equation}
We can also write
\begin{equation}\label{EQ:Lap2}
\L=\Delta_{\widetilde{x}'}+\frac{|\widetilde{x}'|^2}{4}\partial_t^2+Z\partial_t,
\quad \textrm{ with } Z=\sum_{j=1}^n (x_j\partial_{y_j}-y_j\partial_{x_j}),
\end{equation}
where $\Delta_{\widetilde{x}'}$ is the Euclidean Laplacian with respect to ${\widetilde{x}'}$, and $Z$ is the tangential derivative in the $\widetilde{x}'$-variables.

\begin{itemize}
\item Let $\alpha, \beta\in \mathbb{R}$. Then for all complex-valued functions $f\in C_{0}^{\infty}(\mathbb{H}_{n}\backslash\{\widetilde{x}'=0\})$, we have
\begin{equation}\label{EQ:Heis1}
\begin{split}
\|\mathcal{L} f\|^{2}_{L^{2}(\mathbb{H}_{n})}&\geq ((N-4)\alpha-2\beta)\left\|\frac{\nabla_{H} f}{|\widetilde{x}'|}\right\|^{2}_{L^{2}(\mathbb{H}_{n})}\\
&-\alpha(\alpha-4)\left\|\frac{\widetilde{x}'\cdot\nabla_{H} f}{|\widetilde{x}'|^{2}}\right\|^{2}_{L^{2}(\mathbb{H}_{n})}\\
&+\beta((N-4)(\alpha-2)-\beta)\left\|\frac{f}{|\widetilde{x}'|^{2}}\right\|^{2}_{L^{2}(\mathbb{H}_{n})}\\
&+2\alpha
\left(\frac{Zf}{|\widetilde{x}'|}, \frac{Tf}{|\widetilde{x}'|}\right)_{L^{2}(\mathbb{H}_{n})}
+ \alpha \|Tf\|^2_{L^{2}(\mathbb{H}_{n})}
\end{split}
\end{equation}
and
\begin{equation}\label{EQ:Heis11}
\begin{split}
\|\mathcal{L} f\|^{2}_{L^{2}(\mathbb{H}_{n})}&\geq (N\alpha-2\beta)\left\|\frac{\nabla_{H} f}{|\widetilde{x}'|}\right\|^{2}_{L^{2}(\mathbb{H}_{n})}\\
&-\alpha(\alpha+4)\left\|\frac{\widetilde{x}'\cdot\nabla_{H} f}{|\widetilde{x}'|^{2}}\right\|^{2}_{L^{2}(\mathbb{H}_{n})}\\
&+(2(N-4)(\alpha(N-2)-\beta)-2\alpha^{2}(N-2)+\alpha\beta N-\beta^{2})\left\|\frac{f}{|\widetilde{x}'|^{2}}\right\|^{2}_{L^{2}(\mathbb{H}_{n})}\\
&-2\alpha
\left(\frac{Zf}{|\widetilde{x}'|}, \frac{Tf}{|\widetilde{x}'|}\right)_{L^{2}(\mathbb{H}_{n})}
-\alpha \|Tf\|^2_{L^{2}(\mathbb{H}_{n})}.
\end{split}
\end{equation}
Moreover, we have
\begin{equation}\label{EQ:Heis2}
\begin{split}
\|\mathcal{L} f\|^{2}_{L^{2}(\mathbb{H}_{n})}&\geq ((N-\alpha)\alpha-2\beta)\left\|\frac{\nabla_{H} f}{|\widetilde{x}'|}\right\|^{2}_{L^{2}(\mathbb{H}_{n})}\\
&+\beta((N-4)(\alpha-2)-\beta)\left\|\frac{f}{|\widetilde{x}'|^{2}}\right\|^{2}_{L^{2}(\mathbb{H}_{n})}\\
&+ 2\alpha
\left(\frac{Zf}{|\widetilde{x}'|}, \frac{Tf}{|\widetilde{x}'|}\right)_{L^{2}(\mathbb{H}_{n})}
+ \alpha \|Tf\|^2_{L^{2}(\mathbb{H}_{n})}
\end{split}
\end{equation}
for $\alpha(\alpha-4)\geq0$, and
\begin{equation}\label{EQ:Heis22}
\begin{split}
& \|\mathcal{L} f\|^{2}_{L^{2}(\mathbb{H}_{n})}\\
&\geq (-2\beta+\alpha(N-\alpha)-4\alpha)\left\|\frac{\nabla_{H} f}{|\widetilde{x}'|}\right\|^{2}_{L^{2}(\mathbb{H}_{n})}\\
&+(2(N-4)(\alpha(N-2)-\beta)-2\alpha^{2}(N-2)+\alpha\beta N-\beta^{2})\left\|\frac{f}{|\widetilde{x}'|^{2}}\right\|^{2}_{L^{2}(\mathbb{H}_{n})}\\
&-2\alpha
\left(\frac{Zf}{|\widetilde{x}'|}, \frac{Tf}{|\widetilde{x}'|}\right)_{L^{2}(\mathbb{H}_{n})}
- \alpha \|Tf\|^2_{L^{2}(\mathbb{H}_{n})}
\end{split}
\end{equation}
for $\alpha(\alpha+4)\geq0$.
Comparing \eqref{EQ:Heis2} with Gesztesy and Littlejohn's estimate \eqref{Gesztesy_ineq} we note the appearing two last terms involving the commutators $T=[X_{j},Y_{j}]$ due to the non-commutative structure of the Heisenberg group. There are also other ways of writing these inequalities, see \eqref{heisen_Rellich1-1} and \eqref{heisen_Rellich1-11}.

The term $\left(\frac{Zf}{|\widetilde{x}'|}, \frac{Tf}{|\widetilde{x}'|}\right)_{L^{2}(\mathbb{H}_{n})}$, although appearing to be complex-valued, is actually real-valued for any complex-valued function $f$, see Remark \ref{REM:real}, and we have
\begin{equation}\label{EQ:realterm}
\left(\frac{Zf}{|\widetilde{x}'|}, \frac{Tf}{|\widetilde{x}'|}\right)_{L^{2}(\mathbb{H}_{n})}
=-\sum_{j=1}^n\int_{\mathbb{H}_{n}}\log|\widetilde{x}'|^{2}
{\rm Re}(\partial_{y_j}(f(\widetilde{x}',t))\overline{\partial_{tx_j}f(\widetilde{x}',t)}
)d\widetilde{x}' dt.
\end{equation}

For functions $f(\widetilde{x}',t)=f(\widetilde{x}')$ the last two terms in the above inequalities vanish, while for functions $f(\widetilde{x}',t)=f(|\widetilde{x}'|,t)$ we have $Zf=0$, so that the inner product term involving $Zf$ vanishes.

The estimate \eqref{EQ:Heis2} for $\alpha>0$ and \eqref{EQ:Heis22} for $\alpha<0$ show that $ \|Tf\|_{L^{2}(\mathbb{H}_{n})}$ is also controlled by
$\|\mathcal{L} f\|_{L^{2}(\mathbb{H}_{n})}$, although without the weight, which is natural in view of the homogeneity degrees.
\end{itemize}

In Section \ref{SEC:prelim} we briefly recall the main concepts of general homogeneous groups and stratified groups, and fix the notation. In Section \ref{SEC:Fac_hom} the Hardy inequalities with more general weights on homogeneous groups are proved by factorization of differential expressions. In Section \ref{SEC:Fac_strat} we investigate the Hardy and Hardy-Rellich inequalities on stratified Lie groups using this factorization method. Finally, the Hardy-Rellich inequalities on the Heisenberg group are discussed in Section \ref{SEC:Fac_Heisenberg}.

\section{Preliminaries}
\label{SEC:prelim}

In this section we briefly recall the necessary notation concerning the setting of homogeneous
groups following Folland and Stein \cite{FS-Hardy} as well as a recent treatise \cite{FR16}.
A connected simply connected Lie group $\mathbb G$ is called a {\em homogeneous group} if
its Lie algebra $\mathfrak{g}$ is equipped with a family of the following dilations:
$$D_{\lambda}={\rm Exp}(A \,{\rm ln}\lambda)=\sum_{k=0}^{\infty}
\frac{1}{k!}({\rm ln}(\lambda) A)^{k}.$$
Here $A$ is a diagonalisable positive linear operator on Lie algebra $\mathfrak{g}$,
and every dilation $D_{\lambda}$ satisfies the following
$$\forall X,Y\in \mathfrak{g},\, \lambda>0,\;
[D_{\lambda}X, D_{\lambda}Y]=D_{\lambda}[X,Y],$$
that is, every $D_{\lambda}$ is a morphism of the Lie algebra $\mathfrak{g}$.
Then, in particular, we have
\begin{equation}
|D_{\lambda}(S)|=\lambda^{Q}|S| \quad {\rm and}\quad \int_{\mathbb{G}}f(\lambda x)
dx=\lambda^{-Q}\int_{\mathbb{G}}f(x)dx,
\end{equation}
where $Q := {\rm Tr}\,A$ is a homogeneous dimension of $\mathbb G$.
Here $dx$ is the Haar measure on homogeneous group $\mathbb{G}$ and $|S|$ is the volume of a measurable set $S\subset \mathbb{G}$. We recall that the Haar measure on a homogeneous group $\mathbb{G}$ is the standard Lebesgue measure for $\Rn$ (see, for example \cite[Proposition 1.6.6]{FR16}).

Let $|\cdot|$ be a homogeneous quasi-norm on homogeneous groups $\mathbb G$: it satisfies the usual properties of the norm except that the triangle inequality may hold with a constant $\geq 1$, see \cite[Section 3.1.6]{FR16} for a detailed discussion. Let us now introduce the polar decomposition on homogeneous Lie group, which can be found in \cite{FS-Hardy} and \cite[Section 3.1.7]{FR16}:
there is a positive Borel measure $\sigma$ on the unit quasi-sphere
\begin{equation}\label{EQ:sphere}
\wp:=\{x\in \mathbb{G}:\,|x|=1\},
\end{equation}
so that for all $f\in L^{1}(\mathbb{G})$ we have
\begin{equation}\label{EQ:polar}
\int_{\mathbb{G}}f(x)dx=\int_{0}^{\infty}
\int_{\wp}f(ry)r^{Q-1}d\sigma(y)dr.
\end{equation}
If we fix a basis $\{X_{1},\ldots,X_{n}\}$ of a Lie algebra $\mathfrak{g}$
such that
$$AX_{k}=\nu_{k}X_{k}$$
for every $k$, then the matrix $A$ can be taken to be
$A={\rm diag} (\nu_{1},\ldots,\nu_{n})$.
Then each $X_{k}$ is homogeneous of degree $\nu_{k}$. By a decomposition of ${\exp}_{\mathbb{G}}^{-1}(x)$ in $\mathfrak g$, we define the vector
$$e(x)=(e_{1}(x),\ldots,e_{n}(x))$$
by the formula
$${\exp}_{\mathbb{G}}^{-1}(x)=e(x)\cdot \nabla\equiv\sum_{j=1}^{n}e_{j}(x)X_{j},$$
where $\nabla=(X_{1},\ldots,X_{n})$.
It gives the following equality
$$x={\exp}_{\mathbb{G}}\left(e_{1}(x)X_{1}+\ldots+e_{n}(x)X_{n}\right).$$
By homogeneity and denoting $x=ry,\,y\in \wp,$ we get
$$
e(x)=e(ry)=(r^{\nu_{1}}e_{1}(y),\ldots,r^{\nu_{n}}e_{n}(y)).
$$
So one obtains
\begin{equation}\label{dfdr0}
\frac{d}{d|x|}(f(x))=\frac{d}{dr}(f(ry))=
 \frac{d}{dr}(f({\exp}_{\mathbb{G}}
\left(r^{\nu_{1}}e_{1}(y)X_{1}+\ldots
+r^{\nu_{n}}e_{n}(y)X_{n}\right))).
\end{equation}
Throughout this paper, we use the notation
\begin{equation}\label{EQ:Euler}
\mathcal{R} :=\frac{d}{dr},
\end{equation}
that is,
\begin{equation}\label{dfdr}
	\frac{d}{d|x|}(f(x))=\mathcal{R}f(x), \quad\forall x\in \mathbb G,
\end{equation}
for a homogeneous quasi-norm $|x|$ on the homogeneous group $\mathbb G$.

Now we very briefly recall the necessary notation concerning the setting of stratified groups (or homogeneous Carnot groups), as a special case of homogeneous groups.

The triple $\mathbb{G}=(\mathbb{R}^{n}, \circ, D_{\lambda})$ is called a {\em stratified group} if it satisfies the conditions:
\begin{itemize}
\item For some natural numbers $N=N_{1},N_{2},...,N_{r}$ with $N+N_{2}+\ldots+N_{r}=n$, the following decomposition $\mathbb{R}^{n}=\mathbb{R}^{N}\times\ldots\times\mathbb{R}^{N_{r}}$ is valid, and for each $\lambda>0$ the dilation $D_{\lambda}:\mathbb{R}^{n}\rightarrow\mathbb{R}^{n}$ is defined by
    $$D_{\lambda}(x)=D_{\lambda}(x',x^{(2)},\ldots,x^{(r)}):=
    (\lambda x', \lambda^{2}x^{(2)},\ldots,\lambda^{r}x^{(r)})$$ is an automorphism of the stratified group $\mathbb{G}$. Here $x'\equiv x^{(1)}\in\mathbb{R}^{N}$ and $x^{(k)}\in\mathbb{R}^{N_{k}}$ for $k=2,\ldots,r$.
\item Let $N$ be as in above and let $X_{1}, \ldots, X_{N}$ be the left invariant vector fields on stratified group $\mathbb{G}$ such that $X_{k}(0)=\frac{\partial}{\partial x_{k}}|_{0}$ for $k=1, \ldots, N$. Then
    $${\rm rank}({\rm Lie}\{X_{1}, \ldots, X_{N}\})=n,$$
for each $x\in\mathbb{R}^{n}$, that is, the iterated commutators of $X_{1}, \ldots, X_{N}$ span the Lie algebra of stratified group $\mathbb{G}$.
\end{itemize}

Note that the left invariant vector fields $X_{1}, \ldots, X_{N}$ are called the (Jacobian) generators of the stratified group $\mathbb{G}$ and $r$ is called a step of this stratified group $\mathbb{G}$.
For the expressions for left invariant vector fields on $\G$ in terms of the usual (Euclidean) derivatives and further properties see e.g. \cite[Section 3.1.5]{FR16}.

The homogeneous dimension $Q$ of the stratified group $\mathbb{G}$ is then given by
$$Q=\sum_{k=1}^{r}kN_{k}, \;\; N_{1}=N.$$
The differential operator
\begin{equation}\label{L}
\L=\sum_{k=1}^{N}X_{k}^{2},
\end{equation}
is called the (canonical) sub-Laplacian on the stratified group $\G$. This sub-Laplacian $\L$ is a left invariant homogeneous of order 2 hypoelliptic differential operator and it is known that $\L$ is elliptic if and only if $r=1$. The left invariant vector fields $X_{j}$ have an explicit form and satisfy the divergence theorem, which can be found in \cite[Section 3.1.5]{FR16} and \cite{RS17b}:
\begin{equation}\label{Xk}
X_{k}=\frac{\partial}{\partial x'_{k}}+\sum_{\ell=2}^{r}\sum_{m=1}^{N_{1}}a_{k,m}^{(\ell)}
(x', \ldots, x^{\ell-1})\frac{\partial}{\partial x_{m}^{(\ell)}}.
\end{equation}
We will also use the following notations:
$$\nabla_{H}:=(X_{1}, \ldots, X_{N})$$
for the horizontal gradient,
$${\rm div}_{H}\upsilon:=\nabla_{H}\cdot \upsilon$$
for the horizontal divergence, and
$$|x'|=\sqrt{x_{1}^{'2}+\ldots+x_{N}^{'2}}$$
for the Euclidean norm on $\mathbb{R}^{N}$.

Since we have the explicit representation of the left invariant vector fields $X_{j}$ in \eqref{Xk}, one can readily obtain the identities
\begin{equation}\label{formula1}
|\nabla_{H}|x'|^{\gamma}|=\gamma|x'|^{\gamma-1},
\end{equation}
and
\begin{equation}\label{formula2}
{\rm div}_{H}\left(\frac{x'}{|x'|^{\gamma}}\right)=\frac{\sum_{j=1}^{N}|x'|^{\gamma}X_{j}x'_{j}-
\sum_{j=1}^{N}x'_{j}\gamma|x'|^{\gamma-1}X_{j}|x'|}{|x'|^{2\gamma}}=\frac{N-\gamma}{|x'|^{\gamma}},
\end{equation}
for all $\gamma \in \mathbb{R}$, $x'\neq0$.

\section{Factorizations and weighted Hardy inequalities on homogeneous groups}
\label{SEC:Fac_hom}

In this section, using the factorization of differential expressions, we obtain Hardy type inequalities with general weights $\phi(x)$ and $\psi(x)$, which are real-valued functions in $L^2_{loc}(\G\backslash\{0\})$, with their radial derivatives also in $L^2_{loc}(\G\backslash\{0\})$.

The obtained results are new already in the standard setting of $\Rn$, but the proof below works for general homogeneous groups, in particular including $\Rn$ with both isotropic and anisotropic structures, as well as general stratified and graded Lie groups.

\begin{thm}\label{Fac_hom_thm} Let $\mathbb{G}$ be a homogeneous group
of homogeneous dimension $Q\geq 3$ and let $|\cdot|$ be an arbitrary homogeneous quasi-norm. Let $\phi,\psi\in L_{loc}^{2}(\mathbb{G}\backslash\{0\})$ be any real-valued functions such that $\R \phi,\R \psi\in L_{loc}^{2}(\mathbb{G}\backslash\{0\})$. Let $\alpha\in\mathbb{R}$. Then for all complex-valued functions $f\in C_{0}^{\infty}(\mathbb{G}\backslash\{0\})$ we have
$$\int_{\G}(\phi(x))^{2}|\R f(x)|^{2}dx\geq \alpha \int_{\G}\left(\phi(x)\R \psi(x)+\psi(x) \R \phi(x)\right)|f(x)|^{2}dx$$
\begin{equation}\label{fac_hom1}
+\alpha(Q-1)\int_{\G}\frac{\phi(x)\psi(x)}{|x|}|f(x)|^{2}dx-\alpha^{2}\int_{\G}(\psi(x))^{2}|f(x)|^{2}dx
\end{equation}
and
$$\int_{\G}(\phi(x))^{2}|\R f(x)|^{2}dx\geq \alpha \int_{\G}\left(\psi(x) \R \phi(x)-\phi(x)\R \psi(x)\right)|f(x)|^{2}dx$$
$$
+\alpha(Q-1)\int_{\G}\frac{\phi(x)\psi(x)}{|x|}|f(x)|^{2}dx-\alpha^{2}\int_{\G}(\psi(x))^{2}|f(x)|^{2}dx
$$
$$-
(Q-1)\int_{\G}\frac{(\phi(x))^{2}}{|x|^{2}}|f(x)|^{2}dx+(Q-1)\int_{\G}\frac{\phi(x)\R \phi(x)}{|x|}|f(x)|^{2}dx
$$
\begin{equation}\label{fac_hom1_1}+
\int_{\G}\phi(x)\R^{2}\phi(x)|f(x)|^{2}dx.\end{equation}
\end{thm}
From these we can get different weighted Hardy inequalities. For example, in the most physical situation with $\phi(x)\equiv 1$ we obtain for $\alpha\in\mathbb{R}$ and for any $f\in C_{0}^{\infty}(\mathbb{G}\backslash\{0\})$ the inequalities
$$\int_{\G}|\R f(x)|^{2}dx\geq \int_{\G}\left(\alpha \R \phi(x)+\alpha(Q-1)\frac{\psi(x)}{|x|}-\alpha^{2}(\psi(x))^{2}\right)|f(x)|^{2}dx$$
and
$$\int_{\G}|\R f(x)|^{2}dx\geq \int_{\G}\left(-\alpha \R \phi(x)+\alpha(Q-1)\frac{\psi(x)}{|x|}-\alpha^{2}(\psi(x))^{2}-\frac{Q-1}{|x|^{2}}\right)|f(x)|^{2}dx.
$$
\begin{rem} If $\phi(x)=|x|^{-a}$ and $\psi(x)=|x|^{-b}$ for $a,b\in\mathbb{R}$, then \eqref{fac_hom1} implies that
\begin{equation}\label{case1}
\int_{\G}\frac{|\R f(x)|^{2}}{|x|^{2a}}dx\geq \alpha(Q-a-b-1)\int_{\G}\frac{|f(x)|^{2}}{|x|^{a+b+1}}dx-\alpha^{2}\int_{\G}\frac{|f(x)|^{2}}{|x|^{2b}}dx.
\end{equation}
In the case $b=a+1$, we have
$$
\int_{\G}\frac{|\R f(x)|^{2}}{|x|^{2a}}dx\geq (\alpha(Q-2a-2)-\alpha^{2})\int_{\G}\frac{|f(x)|^{2}}{|x|^{2a+2}}dx. $$
Then, by maximising the constant $(\alpha(Q-2a-2)-\alpha^{2})$ with respect to $\alpha$ we obtain the weighted Hardy inequality on homogeneous groups
\begin{equation}\label{fac_hom2}
\int_{\G}\frac{|\R f(x)|^{2}}{|x|^{2a}}dx\geq \frac{(Q-2a-2)^{2}}{4}\int_{\G}\frac{|f(x)|^{2}}{|x|^{2a+2}}dx.
\end{equation}
It is known that the constant in \eqref{fac_hom2} is sharp (see \cite[Corollary 4.2]{RS16} and also \cite[Theorem 3.4]{RSY16}).

Thus, in the case $b=a+1$ and $\alpha=\frac{Q-2a-2}{2}$, we see that \eqref{case1} gives \eqref{fac_hom2}.
\end{rem}
\begin{rem} If $\phi(x)=|x|^{-a}(\log|x|)^{c}$ and $\psi(x)=|x|^{-b}(\log|x|)^{d}$ for $a,b,c,d\in\mathbb{R}$, then we obtain from \eqref{fac_hom1} the inequality
$$
\int_{\G}\frac{(\log|x|)^{2c}}{|x|^{2a}}|\R f(x)|^{2}dx$$$$\geq
\alpha\int_{\G}\left(\frac{(c+d)(\log|x|)^{c+d-1}+(Q-1-a-b)(\log|x|)^{c+d}}{|x|^{a+b+1}}\right)|f(x)|^{2}dx$$$$-\alpha^{2}
\int_{\G}\frac{(\log|x|)^{2d}}{|x|^{2b}}|f(x)|^{2}dx.$$
When $a=\frac{Q-2}{2}$, $b=\frac{Q}{2}$, $c=1$ and $d=0$, it follows that
\begin{equation}\label{case2}
\int_{\G}\frac{(\log|x|)^{2}}{|x|^{Q-2}}|\R f(x)|^{2}dx\geq
(\alpha-\alpha^{2})\int_{\G}\frac{|f(x)|^{2}}{|x|^{Q}}dx,
\end{equation}
which after maximising the above constant with respect to $\alpha$ again, we obtain the critical Hardy inequality
\begin{equation}\label{fac_hom3}
\int_{\G}\frac{(\log|x|)^{2}}{|x|^{Q-2}}|\R f(x)|^{2}dx\geq
\frac{1}{4}\int_{\G}\frac{|f(x)|^{2}}{|x|^{Q}}dx.
\end{equation}
The sharpness of the constant in \eqref{fac_hom3} is proved in \cite[Theorem 3.4]{RSY16}). So, we note that \eqref{case2} gives \eqref{fac_hom3} when $\alpha=\frac{1}{2}$.
\end{rem}
\begin{rem} On Carnot groups, we can refer to the recent work \cite{YKG17} for the Hardy inequalities with a pair of nonnegative weight
functions. However, we note that in Theorem \ref{Fac_hom_thm} there are no restrictions on the weights while in \cite{YKG17} the weights have to satisfy certain relations for the weighted estimate to hold true.
\end{rem}
\begin{proof}[Proof of Theorem \ref{Fac_hom_thm}] Let us introduce one-parameter differential expression
$$T_{\alpha}:=\phi(x)\R+\alpha\psi(x).$$
Let us calculate a formal adjoint operator of $T_{\alpha}$ on $C_{0}^{\infty}(\mathbb{G}\backslash\{0\})$:
$$\int_{\G}\phi(x)\R f(x)\overline{g(x)}dx+\alpha\int_{\G}\psi(x)f(x)\overline{g(x)}dx$$
$$=\int_{0}^{\infty}\int_{\wp}\phi(ry)\frac{d}{dr}(f(ry))\overline{g(ry)}r^{Q-1}d\sigma(y)dr+
\alpha\int_{\G}\psi(x)f(x)\overline{g(x)}dx$$
$$=-\int_{0}^{\infty}\int_{\wp}\phi(ry)f(ry)\overline{\frac{d}{dr}(g(ry))}r^{Q-1}d\sigma(y)dr$$
$$-(Q-1)\int_{0}^{\infty}\int_{\wp}\phi(ry)f(ry)\overline{g(ry)}r^{Q-2}d\sigma(y)dr$$
$$-\int_{0}^{\infty}\int_{\wp}\frac{d}{dr}(\phi(ry))f(ry)\overline{g(ry)}r^{Q-1}d\sigma(y)dr$$
$$+\alpha\int_{\G}\psi(x)f(x)\overline{g(x)}dx$$
$$=\int_{\G}f(x)\left(-\phi(x)\overline{\R g(x)}\right)dx-(Q-1)\int_{\G}f(x)\left(\frac{\phi(x)}{|x|}\overline{g(x)}\right)dx$$
$$-\int_{\G}f(x)\left(\R \phi(x)\overline{g(x)}\right)dx+\alpha\int_{\G}f(x)\left(\psi(x)\overline{g(x)}\right)dx.$$
Thus, the formal adjoint operator of $T_{\alpha}$ has the following form
$$T_{\alpha}^{+}:=-\phi(x)\R-\frac{Q-1}{|x|}\phi(x)-\R \phi(x)+\alpha \psi(x),$$
where $x\neq0$.
Then we have
$$(T_{\alpha}^{+}T_{\alpha}f)(x)=-\phi(x)\R (\phi(x)\R f(x))-\alpha\phi(x)\R (f(x)\psi(x))-\frac{Q-1}{|x|}(\phi(x))^{2}\R f(x)$$
$$-\frac{\alpha(Q-1)}{|x|}\phi(x)\psi(x)f(x)-\phi(x)\R \phi(x)\R f(x)-\alpha\psi(x)f(x)\R \phi(x)$$$$+\alpha\phi(x)\psi(x)\R f(x)+\alpha^{2}(\psi(x))^{2}f(x).$$
By the nonnegativity of $T_{\alpha}^{+}T_{\alpha}$, introducing polar coordinates $(r,y)=(|x|, \frac{x}{\mid x\mid})\in (0,\infty)\times\wp$ on $\mathbb{G}$, where $\wp$ is the quasi-sphere as in \eqref{EQ:sphere}, and using \eqref{EQ:polar}, one calculates
$$0\leq \int_{\G}|(T_{\alpha}f)(x)|^{2}dx=
\int_{\G}f(x)\overline{(T_{\alpha}^{+}T_{\alpha}f)(x)}dx$$
\begin{equation}\label{fac_hom4}
={\rm Re}\int_{\G}f(x)\overline{(T_{\alpha}^{+}T_{\alpha}f)(x)}dx
=I_{1}+I_{2}+I_{3}+I_{4}+I_{5}+I_{6},
\end{equation}
where
$$I_{1}=-{\rm Re}\int_{0}^{\infty}\int_{\wp}f(ry)\phi(ry)\overline{\frac{d}{dr}\left(\phi(ry)\frac{d}{dr} (f(ry))\right)}r^{Q-1}d\sigma(y)dr,$$
$$I_{2}=-\alpha{\rm Re}\int_{0}^{\infty}\int_{\wp}f(ry)\phi(ry)\overline{\frac{d}{dr}(f(ry)\psi(ry))}r^{Q-1}d\sigma(y)dr,$$
$$I_{3}=-(Q-1){\rm Re}\int_{0}^{\infty}\int_{\wp}f(ry)\frac{(\phi(ry))^{2}\overline{\frac{d}{dr} (f(ry))}}{r}r^{Q-1}d\sigma(y)dr,$$
$$I_{4}=-\alpha(Q-1)\int_{\G}\frac{\phi(x)\psi(x)}{|x|}|f(x)|^{2}dx-\alpha\int_{\G}\R \phi(x) \psi(x)|f(x)|^{2}dx$$
\begin{equation}\label{fac_hom_I4}
+\alpha^{2}\int_{\G}(\psi(x))^{2}|f(x)|^{2}dx,
\end{equation}
$$I_{5}=-{\rm Re}\int_{0}^{\infty}\int_{\wp}f(ry)\frac{d}{dr}(\phi(ry))\phi(ry)\overline{\frac{d}{dr}(f(ry))} r^{Q-1}d\sigma(y)dr$$
and
$$I_{6}=\alpha{\rm Re}\int_{0}^{\infty}\int_{\wp}f(ry)\phi(ry)\psi(ry)\overline{\frac{d}{dr}(f(ry))} r^{Q-1}d\sigma(y)dr.$$
Now we calculate $I_{1}$, $I_{2}$, $I_{3}$, $I_{5}$ and $I_{6}$. By a direct calculation we obtain
$$I_{1}=(Q-1){\rm Re}\int_{0}^{\infty}\int_{\wp}(\phi(ry))^{2}{f(ry)}\overline{\frac{d}{dr}(f(ry))}r^{Q-2}d\sigma(y)dr$$
$$+{\rm Re}\int_{0}^{\infty}\int_{\wp}(\phi(ry))^{2}\frac{d}{dr}(f(ry))\overline{\frac{d}{dr}(f(ry))}r^{Q-1}d\sigma(y)dr$$
$$+{\rm Re}\int_{0}^{\infty}\int_{\wp}\phi(ry)\frac{d}{dr}(\phi(ry))\overline{\frac{d}{dr}f(ry)}f(ry)r^{Q-1}d\sigma(y)dr$$
$$=\frac{Q-1}{2}\int_{0}^{\infty}\int_{\wp}(\phi(ry))^{2}\frac{d}{dr}|f(ry)|^{2}r^{Q-2}d\sigma(y)dr$$
$$+\int_{0}^{\infty}\int_{\wp}(\phi(ry))^{2}\left|\frac{d}{dr}(f(ry))\right|^{2}r^{Q-1}d\sigma(y)dr$$
$$+\frac{1}{2}\int_{0}^{\infty}\int_{\wp}\phi(ry)\frac{d}{dr}(\phi(ry))\frac{d}{dr}|f(ry)|^{2}r^{Q-1}d\sigma(y)dr$$
$$=\int_{0}^{\infty}\int_{\wp}(\phi(ry))^{2}\left|\frac{d}{dr}(f(ry))\right|^{2}r^{Q-1}d\sigma(y)dr$$
$$-\frac{(Q-1)(Q-2)}{2}\int_{0}^{\infty}\int_{\wp}(\phi(ry))^{2}|f(ry)|^{2}r^{Q-3}d\sigma(y)dr$$
$$-(Q-1)\int_{0}^{\infty}\int_{\wp}\phi(ry)\frac{d}{dr}(\phi(ry))|f(ry)|^{2}r^{Q-2}d\sigma(y)dr$$
$$-\frac{1}{2}\int_{0}^{\infty}\int_{\wp}\left(\frac{d}{dr}(\phi(ry))\right)^{2}|f(ry)|^{2}r^{Q-1}d\sigma(y)dr$$
$$-\frac{1}{2}\int_{0}^{\infty}\int_{\wp}\phi(ry)\frac{d^{2}}{dr^{2}}(\phi(ry))|f(ry)|^{2}r^{Q-1}d\sigma(y)dr$$
$$-\frac{Q-1}{2}\int_{0}^{\infty}\int_{\wp}\phi(ry)\frac{d}{dr}(\phi(ry))|f(ry)|^{2}r^{Q-2}d\sigma(y)dr$$
$$=\int_{\G}(\phi(x))^{2}|\R f(x)|^{2}dx-\frac{(Q-1)(Q-2)}{2}\int_{\G}\frac{(\phi(x))^{2}}{|x|^{2}}|f(x)|^{2}dx$$
$$-(Q-1)\int_{\G}\frac{\phi(x)\R \phi(x)}{|x|}|f(x)|^{2}dx
-\frac{1}{2}\int_{\G}(\R \phi(x))^{2}|f(x)|^{2}dx$$
\begin{equation}\label{fac_hom_I1}
-\frac{1}{2}\int_{\G}\R^{2}\phi(x)\phi(x)|f(x)|^{2}dx-\frac{Q-1}{2}\int_{\G}\frac{\phi(x)\R \phi(x)}{|x|}|f(x)|^{2}dx.\end{equation}
Now let us calculate $I_{2}$:
$$I_{2}=-\alpha{\rm Re}\int_{0}^{\infty}\int_{\wp}\phi(ry)\psi(ry)f(ry)\overline{\frac{d}{dr}(f(ry))}r^{Q-1}d\sigma(y)dr$$
$$-\alpha\int_{0}^{\infty}\int_{\wp}\phi(ry)\frac{d}{dr}(\psi(ry))|f(ry)|^{2}r^{Q-1}d\sigma(y)dr$$
$$=-\frac{\alpha}{2}\int_{0}^{\infty}\int_{\wp}\phi(ry)\psi(ry)\frac{d}{dr}|f(ry)|^{2}r^{Q-1}d\sigma(y)dr$$
$$-\alpha\int_{0}^{\infty}\int_{\wp}\phi(ry)\frac{d}{dr}(\psi(ry))|f(ry)|^{2}r^{Q-1}d\sigma(y)dr$$
$$=-\alpha\int_{0}^{\infty}\int_{\wp}\phi(ry)\frac{d}{dr}(\psi(ry))|f(ry)|^{2}r^{Q-1}d\sigma(y)dr$$
$$+\frac{\alpha}{2}\int_{0}^{\infty}\int_{\wp}\psi(ry)\frac{d}{dr}(\phi(ry))|f(ry)|^{2}r^{Q-1}d\sigma(y)dr$$
$$+\frac{\alpha}{2}\int_{0}^{\infty}\int_{\wp}\phi(ry)\frac{d}{dr}(\psi(ry))|f(ry)|^{2}r^{Q-1}d\sigma(y)dr$$
$$+\frac{\alpha(Q-1)}{2}\int_{0}^{\infty}\int_{\wp}\phi(ry)\psi(ry)|f(ry)|^{2}r^{Q-2}d\sigma(y)dr$$
$$=\frac{\alpha}{2}\int_{\G}\psi(x)\R\phi(x)|f(x)|^{2}dx+\frac{\alpha}{2}\int_{\G}\phi(x)\R\psi(x)|f(x)|^{2}dx$$
\begin{equation}\label{fac_hom_I2}
+\frac{\alpha(Q-1)}{2}\int_{\G}\frac{\phi(x)\psi(x)}{|x|}|f(x)|^{2}dx-\alpha\int_{\G}\phi(x)\R \psi(x)|f(x)|^{2}dx.
\end{equation}
For $I_{3}$, one has
$$I_{3}=-(Q-1){\rm Re}\int_{0}^{\infty}\int_{\wp}(\phi(ry))^{2}f(ry)\overline{\frac{d}{dr}(f(ry))}r^{Q-2}d\sigma(y)dr$$
$$=-\frac{Q-1}{2}\int_{0}^{\infty}\int_{\wp}(\phi(ry))^{2}\frac{d}{dr}|f(ry)|^{2}r^{Q-2}d\sigma(y)dr$$
$$=(Q-1)\int_{0}^{\infty}\int_{\wp}\phi(ry)\frac{d}{dr}(\phi(ry))|f(ry)|^{2}r^{Q-2}d\sigma(y)dr$$
$$+\frac{(Q-1)(Q-2)}{2}\int_{0}^{\infty}\int_{\wp}(\phi(ry))^{2}|f(ry)|^{2}r^{Q-3}d\sigma(y)dr$$
\begin{equation}\label{fac_hom_I3}
=(Q-1)\int_{\G}\frac{\phi(x)\R \phi(x)}{|x|}|f(x)|^{2}dx+
\frac{(Q-1)(Q-2)}{2}\int_{\G}\frac{(\phi(x))^{2}}{|x|^{2}}|f(x)|^{2}dx.
\end{equation}
For $I_{5}$, we have
$$I_{5}=-{\rm Re}\int_{0}^{\infty}\int_{\wp}f(ry)\frac{d}{dr}(\phi(ry))\phi(ry)\overline{\frac{d}{dr}(f(ry))} r^{Q-1}d\sigma(y)dr$$
$$=-\frac{1}{2}\int_{0}^{\infty}\int_{\wp}\frac{d}{dr}(\phi(ry)) \phi(ry)\frac{d}{dr}|f(ry)|^{2} r^{Q-1}d\sigma(y)dr$$
$$=\frac{1}{2}\int_{0}^{\infty}\int_{\wp}\phi(ry)\frac{d^{2}}{dr^{2}}(\phi(ry))|f(ry)|^{2}r^{Q-1}d\sigma(y)dr$$
$$+\frac{1}{2}\int_{0}^{\infty}\int_{\wp}\left(\frac{d}{dr}(\phi(ry))\right)^{2}|f(ry)|^{2}r^{Q-1}d\sigma(y)dr$$
$$+\frac{Q-1}{2}\int_{0}^{\infty}\int_{\wp}\frac{\phi(ry) \frac{d}{dr}(\phi(ry))}{r}|f(ry)|^{2}r^{Q-1}d\sigma(y)dr$$
$$=\frac{1}{2}\int_{\G}\R^{2}\phi(x)\phi(x)|f(x)|^{2}dx+\frac{1}{2}\int_{\G}(\R \phi(x))^{2}|f(x)|^{2}dx$$
\begin{equation}\label{fac_hom_I5}
+\frac{Q-1}{2}\int_{\G}\frac{\R \phi(x) \phi(x)}{|x|}|f(x)|^{2}dx.
\end{equation}
Finally, for $I_{6}$ we obtain
$$I_{6}=\alpha{\rm Re}\int_{0}^{\infty}\int_{\wp}f(ry)\phi(ry)\psi(ry)\overline{\frac{d}{dr}(f(ry))} r^{Q-1}d\sigma(y)dr$$
$$=\frac{\alpha}{2}\int_{0}^{\infty}\int_{\wp}\phi(ry)\psi(ry)\frac{d}{dr}|f(ry)|^{2} r^{Q-1}d\sigma(y)dr$$
$$=-\frac{\alpha}{2}\int_{0}^{\infty}\int_{\wp}\frac{d}{dr}(\phi(ry))\psi(ry)|f(ry)|^{2}r^{Q-1}d\sigma(y)dr$$
$$-\frac{\alpha}{2}\int_{0}^{\infty}\int_{\wp}\frac{d}{dr}(\psi(ry))\phi(ry)|f(ry)|^{2}r^{Q-1}d\sigma(y)dr$$
$$-\frac{\alpha(Q-1)}{2}\int_{0}^{\infty}\int_{\wp}\phi(ry)\psi(ry)\frac{|f(ry)|^{2}}{r}r^{Q-1}d\sigma(y)dr$$
$$=-\frac{\alpha}{2}\int_{\G}\R \phi(x) \psi(x)|f(x)|^{2}dx-\frac{\alpha}{2}\int_{\G}\phi(x)\R\psi(x)|f(x)|^{2}dx$$
\begin{equation}\label{fac_hom_I6}-\frac{\alpha(Q-1)}{2}\int_{\G}\frac{\phi(x)\psi(x)}{|x|}|f(x)|^{2}dx.
\end{equation}
Putting \eqref{fac_hom_I4}-\eqref{fac_hom_I6} in \eqref{fac_hom4}, we obtain that
$$\int_{\G}(\phi(x))^{2}|\R f(x)|^{2}dx$$$$-\alpha \int_{\G}\left(\phi(x)\R \psi(x)+\psi(x) \R \phi(x)+(Q-1)\frac{\phi(x)\psi(x)}{|x|}\right)|f(x)|^{2}dx$$
$$
+\alpha^{2}\int_{\G}(\psi(x))^{2}|f(x)|^{2}dx\geq0,$$
which implies \eqref{fac_hom1}.

Thus, we have obtained \eqref{fac_hom1} using the nonnegativity of $T_{\alpha}^{+}T_{\alpha}$. Now we obtain \eqref{fac_hom1_1} using the nonnegativity of $T_{\alpha}T_{\alpha}^{+}$. So, we calculate
$$(T_{\alpha}T_{\alpha}^{+}f)(x)=-\phi(x)\R (\phi(x)\R f(x))+\alpha\phi(x)\R (f(x)\psi(x))-\frac{Q-1}{|x|}(\phi(x))^{2}\R f(x)$$
$$-\frac{\alpha(Q-1)}{|x|}\phi(x)\psi(x)f(x)-\phi(x)\R \phi(x)\R f(x)-\alpha\psi(x)f(x)\R \phi(x)$$$$-\alpha\phi(x)\psi(x)\R f(x)+\alpha^{2}(\psi(x))^{2}f(x)$$
$$-(Q-1)\phi(x)f(x)\R \left(\frac{\phi(x)}{|x|}\right)-\phi(x)f(x)\R^{2}\phi(x).$$
Using the nonnegativity of $T_{\alpha}T_{\alpha}^{+}$, we get
$$0\leq \int_{\G}|(T_{\alpha}f)(x)|^{2}dx=
\int_{\G}f(x)\overline{(T_{\alpha}T^{+}_{\alpha}f)(x)}dx$$
\begin{equation}\label{fac_hom4_1}
={\rm Re}\int_{\G}f(x)\overline{(T_{\alpha}T^{+}_{\alpha}f)(x)}dx
=\widetilde{I}_{1}+\widetilde{I}_{2}+\widetilde{I}_{3}+\widetilde{I}_{4}+\widetilde{I}_{5}+\widetilde{I}_{6},
\end{equation}
where
$$\widetilde{I}_{1}=I_{1}, \widetilde{I}_{2}=-I_{2},\widetilde{I}_{3}=I_{3},\widetilde{I}_{4}=I_{4}+A,\widetilde{I}_{5}=I_{5},\widetilde{I}_{6}=-I_{6}$$
with
$$A=-(Q-1)\int_{\G}\phi(x)|f(x)|^{2}\R \left(\frac{\phi(x)}{|x|}\right)dx-\int_{\G}\phi(x)\R^{2}\phi(x)|f(x)|^{2}dx.$$
Taking into account these and \eqref{fac_hom_I4}-\eqref{fac_hom_I6}, we obtain \eqref{fac_hom1_1}.
\end{proof}
\section{Factorizations and Hardy-Rellich inequalities on stratified groups}
\label{SEC:Fac_strat}
In this section, we obtain Hardy-Rellich and improved Hardy inequalities on stratified groups by factorization. We refer to \cite{G84}, \cite{GP80} and to the recent paper \cite{GL17} for obtaining Hardy and Hardy-Rellich inequalities using this factorization method in the isotropic abelian case.

In this section $\L$ is the sub-Laplacian operator defined by \eqref{L}.

\begin{thm}\label{strat_Rellich_ab} Let $\mathbb{G}$ be a stratified group with $N\geq2$ being the dimension of the first stratum, and let $\alpha, \beta\in \mathbb{R}$. Then for all complex-valued functions $f\in C_{0}^{\infty}(\G\backslash\{x'=0\})$ we have
\begin{equation}\label{strat_Rellich1_ab}
\begin{split}
& \|\mathcal{L} f\|^{2}_{L^{2}(\G)}\\
&\geq (\alpha(N-2)-2\beta)\left\|\frac{\nabla_{H} f}{|x'|}\right\|^{2}_{L^{2}(\G)}\\
&-\alpha^{2}\left\|\frac{x'\cdot\nabla_{H} f}{|x'|^{2}}\right\|^{2}_{L^{2}(\G)}\\
&+(\alpha(N-4)(N-2)-\alpha^{2}(N-2)+2\beta(4-N)-\beta^{2}+\alpha\beta(N-2))
\left\|\frac{f}{|x'|^{2}}\right\|^{2}_{L^{2}(\G)}.
\end{split}
\end{equation}
\end{thm}

We note that the term $x'\cdot\nabla_{H} f$ can be eliminated.

\begin{cor} Using the Cauchy-Schwarz inequality in \eqref{strat_Rellich1_ab}
\begin{equation}\label{Cauchy-Schwarz}
\int_{\G}\frac{|x'\cdot (\nabla_{H}f)(x)|^{2}}{|x'|^{4}}dx\leq
\int_{\G}\frac{|(\nabla_{H}f)(x)|^{2}}{|x'|^{2}}dx,
\end{equation}
we obtain
\begin{equation}\label{strat_Rellich2_ab}
\begin{split}
& \|\mathcal{L} f\|^{2}_{L^{2}(\G)}\\
&\geq (\alpha(N-2)-2\beta-\alpha^{2})\left\|\frac{\nabla_{H} f}{|x'|}\right\|^{2}_{L^{2}(\G)}\\
&+(\alpha(N-4)(N-2)-\alpha^{2}(N-2)+2\beta(4-N)-\beta^{2}+\alpha\beta(N-2))
\left\|\frac{f}{|x'|^{2}}\right\|^{2}_{L^{2}(\G)}.
\end{split}
\end{equation}
\end{cor}

We also record the special case of Theorem \ref{strat_Rellich_ab} in the abelian estting, to contrast it later in Corollary \ref{cor_Rn2} with the Gesztesy-Littlejohn inequality \eqref{Gesztesy_ineq}.

\begin{cor}\label{cor_Rn1} In the abelian case $\G =(\Rn, +)$, we have $N=n$, $\nabla_{H}=\nabla=(\partial_{x_{1}},...,\partial_{x_{n}})$ is the usual (full) gradient, so \eqref{strat_Rellich1_ab} implies for $\alpha,\beta\in \mathbb{R}$ and for any $f\in C_{0}^{\infty}(\Rn\backslash\{0\})$ with $n\geq2$ the inequality
\begin{equation}\label{strat_Rellich1_ab_Rn}
\begin{split}
& \|\triangle f\|^{2}_{L^{2}(\Rn)}\\
&\geq (\alpha(n-2)-2\beta)\left\|\frac{\nabla f}{|x|}\right\|^{2}_{L^{2}(\Rn)}\\
&-\alpha^{2}\left\|\frac{x\cdot\nabla f}{|x|^{2}}\right\|^{2}_{L^{2}(\Rn)}\\
&+(\alpha(n-4)(n-2)-\alpha^{2}(n-2)+2\beta(4-n)-\beta^{2}+\alpha\beta(n-2))
\left\|\frac{f}{|x|^{2}}\right\|^{2}_{L^{2}(\Rn)}.
\end{split}
\end{equation}
\end{cor}
\begin{proof}[Proof of Theorem \ref{strat_Rellich_ab}]
We will be applying the factorization method with the following differential expressions for two real parameters $\alpha,\beta\in\mathbb R$,
\begin{equation}\label{dif_exp1_ab}
T_{\alpha, \beta}:=-\mathcal{L}+\alpha\frac{x'\cdot\nabla_{H}}{|x'|^{2}}+\frac{\beta}{|x'|^{2}}
\end{equation}
and its formal adjoint
\begin{equation}\label{dif_exp2_ab}T^{+}_{\alpha,\beta}:=-\mathcal{L}-\alpha\frac{x'\cdot\nabla_{H}}{|x'|^{2}}
-\frac{\alpha(N-2)-\beta}{|x'|^{2}},
\end{equation}
where $x'\neq0$. Then, by a direct calculation for any function $f\in C_{0}^{\infty}(\G\backslash\{x'=0\})$ we have
\begin{equation}\label{strat_TT0_ab}
\begin{aligned}
& (T^{+}_{\alpha,\beta}T_{\alpha,\beta}f)(x) \\
& =\left(-\mathcal{L}-\alpha\frac{x'\cdot\nabla_{H}}{|x'|^{2}}
-\frac{\alpha(N-2)-\beta}{|x'|^{2}}\right)\left(-(\mathcal{L}f)(x)
+\alpha\frac{x'\cdot(\nabla_{H}f)}{|x'|^{2}}+
\frac{\beta f(x)}{|x'|^{2}}\right) \\
& =(\L^{2}f)(x)+\alpha\left(-\L\left(\frac{x'\cdot(\nabla_{H}f)}{|x'|^{2}}\right)(x)+\frac{x'\cdot\nabla_{H}}{|x'|^{2}}(\L f)(x)+
\frac{N-2}{|x'|^{2}}(\L f)(x)\right) \\
& \quad +\beta\left(-\L\left(\frac{f}{|x'|^{2}}\right)(x)-\frac{(\L f)(x)}{|x'|^{2}}\right) \\
& \quad +\alpha\beta\left(-\frac{x'\cdot\nabla_{H}}{|x'|^{2}}\left(\frac{f}{|x'|^{2}}\right)(x)+
\frac{x'\cdot(\nabla_{H}f)(x)}{|x'|^{4}}-\frac{(N-2)f(x)}{|x'|^{4}}\right)
\\
& \; +\alpha^{2}\left(-\left(\frac{x'\cdot\nabla_{H}}{|x'|^{2}}\right)
\left(\frac{x'\cdot(\nabla_{H}f)}{|x'|^{2}}\right)(x)-
(N-2)\frac{x'\cdot(\nabla_{H}f)(x)}{|x'|^{4}}\right)+\beta^{2}\frac{f(x)}{|x'|^{4}}.
\end{aligned}
\end{equation}
Now we calculate
\begin{equation}\label{strat_TT*0_ab}
\begin{aligned}
& (T_{\alpha,\beta}T^{+}_{\alpha,\beta}f)(x) \\
& =\left(-\mathcal{L}
+\alpha\frac{x'\cdot\nabla_{H}}{|x'|^{2}}+
\frac{\beta }{|x'|^{2}}\right)\left(-(\mathcal{L}f)(x)-\alpha\frac{x'\cdot(\nabla_{H}f)(x)}{|x'|^{2}}
-\frac{(\alpha(N-2)-\beta)f(x)}{|x'|^{2}}\right)\\
& =(\L^{2}f)(x)+
\alpha\left(\L\left(\frac{x'\cdot(\nabla_{H}f)}{|x'|^{2}}\right)(x)
-\frac{x'\cdot\nabla_{H}}{|x'|^{2}}(\L f)(x)+
(N-2)\L\left(\frac{f}{|x'|^{2}}\right)(x)\right) \\
& \quad +\beta\left(-\L\left(\frac{f}{|x'|^{2}}\right)(x)-\frac{(\L f)(x)}{|x'|^{2}}\right) \\
& \quad +\alpha\beta\left(\frac{x'\cdot\nabla_{H}}{|x'|^{2}}\left(\frac{f}{|x'|^{2}}\right)(x)-
\frac{x'\cdot(\nabla_{H}f)(x)}{|x'|^{4}}-\frac{(N-2)f(x)}{|x'|^{4}}\right)
\\
& \; +\alpha^{2}\left(-\left(\frac{x'\cdot\nabla_{H}}{|x'|^{2}}\right)
\left(\frac{x'\cdot(\nabla_{H}f)}{|x'|^{2}}\right)(x)-
(N-2)\frac{x'\cdot\nabla_{H}}{|x'|^{2}}\left(\frac{f}{|x'|^{2}}\right)(x)\right)
+\beta^{2}\frac{f(x)}{|x'|^{4}}.
\end{aligned}
\end{equation}
Using
\begin{equation}\label{formula3}
\begin{split}
\L\left(\frac{f}{|x'|^{2}}\right)(x)&=
\sum_{j=1}^{N}X^{2}_{j}\left(\frac{f}{|x'|^{2}}\right)(x)\\&=
\sum_{j=1}^{N}X_{j}\left(\frac{X_{j}f}{|x'|^{2}}
-\frac{2x'_{j}f}{|x'|^{4}}\right)(x)\\&
=\sum_{j=1}^{N}\left(\frac{(X^{2}_{j}f)(x)}{|x'|^{2}}-
\frac{4x'_{j}(X_{j}f)(x)}{|x'|^{4}}
+\frac{8(x'_{j})^{2}f(x)}{|x'|^{6}}
-\frac{2f(x)}{|x'|^{4}}\right)\\&
=\frac{(\L f)(x)}{|x'|^{2}}
-\frac{4x'\cdot(\nabla_{H}f)(x)}{|x'|^{4}}
-(2N-8)\frac{f(x)}{|x'|^{4}}
\end{split}
\end{equation}
and
\begin{equation}\label{formula4}
\begin{split}
\frac{x'\cdot\nabla_{H}}{|x'|^{2}}\left(\frac{f}{|x'|^{2}}\right)(x)
&=\frac{f(x)}{|x'|^{2}}\sum_{j=1}^{N}x'_{j}X_{j}(|x'|^{-2})
+\frac{x'\cdot(\nabla_{H}f)(x)}{|x'|^{4}}
\\&=-2\sum_{j=1}^{N}\frac{(x'_{j})^{2}f(x)}{|x'|^{6}}
+\frac{x'\cdot(\nabla_{H}f)(x)}{|x'|^{4}}
\\&=-2\frac{f(x)}{|x'|^{4}}+\frac{x'\cdot(\nabla_{H}f)(x)}{|x'|^{4}},
\end{split}
\end{equation}
in \eqref{strat_TT*0_ab}, then we have for $(T^{+}_{\alpha,\beta}T_{\alpha,\beta}f)(x)+(T_{\alpha,\beta}T^{+}_{\alpha,\beta}f)(x)$ that
\begin{equation}\label{strat_TT*01_ab}
\begin{split}
& (T^{+}_{\alpha,\beta}T_{\alpha,\beta}f)(x)+(T_{\alpha,\beta}T^{+}_{\alpha,\beta}f)(x) \\
& =2(\L^{2}f)(x)+
2\alpha(N-2)\left(\frac{(\L f)(x)}{|x'|^{2}}-2\frac{x'\cdot(\nabla_{H}f)(x)}{|x'|^{4}}
+(4-N)\frac{f(x)}{|x'|^{4}}\right)\\
& \quad +2\beta\left(-\frac{2(\L f)(x)}{|x'|^{2}}+\frac{4x'\cdot(\nabla_{H}f)(x)}{|x'|^{4}}+(2N-8)\frac{f(x)}{|x'|^{4}}\right) -2\alpha\beta(N-2)\frac{f(x)}{|x'|^{4}}
\\
& \; +2\alpha^{2}\left(-\left(\frac{x'\cdot\nabla_{H}}{|x'|^{2}}\right)
\left(\frac{x'\cdot(\nabla_{H}f)}{|x'|^{2}}\right)(x)-
(N-2)\frac{x'\cdot(\nabla_{H}f)(x)}{|x'|^{4}}+(N-2)\frac{f(x)}{|x'|^{4}}\right)\\
& \;
+2\beta^{2}\frac{f(x)}{|x'|^{4}} \\
& =:2(\L^{2}f)(x)+2\alpha(N-2)\left(\frac{(\L f)(x)}{|x'|^{2}}-2\frac{x'\cdot(\nabla_{H}f)(x)}{|x'|^{4}}
+(4-N)\frac{f(x)}{|x'|^{4}}\right)\\&
-2\alpha\beta(N-2)\frac{f(x)}{|x'|^{4}}+2\beta^{2}\frac{f(x)}{|x'|^{4}}+\beta J_{1}+\alpha^{2}J_{2}.
\end{split}
\end{equation}
In order to simplify this, let us calculate the following
\begin{equation}\label{J2_ab}
\begin{split}
&-2\left(\frac{x'\cdot\nabla_{H}}{|x'|^{2}}\right)
\left(\frac{x'\cdot(\nabla_{H}f)}{|x'|^{2}}\right)(x)-2
(N-2)\frac{x'\cdot(\nabla_{H}f)(x)}{|x'|^{4}}+2(N-2)\frac{f(x)}{|x'|^{4}} \\
& =-\frac{2\sum_{j,k=1}^{N}(x'_{j}X_{j})
(x'_{k}(X_{k}f))(x)}{|x'|^{4}}-2
\sum_{j,k=1}^{N}x'_{j}(-2)|x'|^{-3}X_{j}|x'|
\frac{x'_{k}(X_{k}f)(x)}{|x'|^{2}} \\
& \quad -
2(N-2)\frac{x'\cdot(\nabla_{H}f)(x)}{|x'|^{4}}+2(N-2)\frac{f(x)}{|x'|^{4}}\\
& =-\frac{2\sum_{k=1}^{N}x'_{k}(X_{k}f)(x)}{|x'|^{4}}
-\frac{2\sum_{j,k=1}^{N}x'_{j}x'_{k}X_{j}(X_{k}f)(x)}{|x'|^{4}}
+\frac{4\sum_{k=1}^{N}x'_{k}(X_{k}f)(x)}{|x'|^{4}} \\
& \quad -
2(N-2)\frac{x'\cdot(\nabla_{H}f)(x)}{|x'|^{4}}+2(N-2)\frac{f(x)}{|x'|^{4}} \\
& =-
2(N-3)\frac{x'\cdot(\nabla_{H}f)(x)}{|x'|^{4}}-
\frac{2\sum_{j,k=1}^{N}x'_{j}x'_{k}(X_{j}X_{k}f)(x)}{|x'|^{4}}
+2(N-2)\frac{f(x)}{|x'|^{4}}.
\end{split}
\end{equation}
Now putting this in \eqref{strat_TT*01_ab}, we obtain
\begin{equation}\label{strat_TT1_ab}
\begin{split}
 (T^{+}_{\alpha,\beta}T_{\alpha,\beta}f)(x)+(T_{\alpha,\beta}T^{+}_{\alpha,\beta}f)(x) =&2(\L^{2}f)(x)+(2\alpha(N-2)-4\beta)\frac{(\L f)(x)}{|x'|^{2}}\\
&+(-4\alpha(N-2)-2\alpha^{2}(N-3)+8\beta)\frac{x'\cdot(\nabla_{H} f)(x)}{|x'|^{4}}\\
&+(2\alpha(N-2)(4-N)+2\alpha^{2}(N-2)-2\alpha\beta(N-2)\\&+(4N-16)\beta+2\beta^{2})\frac{f(x)}{|x'|^{4}}\\
&-2\alpha^{2}\frac{\sum_{j,k=1}^{N}x'_{j}x'_{k}(X_{j}X_{k}f)(x)}{|x'|^{4}}.
\end{split}
\end{equation}
Now using the nonnegativity of $T_{\alpha,\beta}^{+}T_{\alpha,\beta}+T_{\alpha,\beta}T^{+}_{\alpha,\beta}$ and integrating by parts, we get
$$\int_{\G}|(T_{\alpha,\beta}f)(x)|^{2}dx+\int_{\G}|(T^{+}_{\alpha,\beta}f)(x)|^{2}dx$$
$$=\int_{\G}\overline{f(x)}((T^{+}_{\alpha,\beta}T_{\alpha,\beta}+T_{\alpha,\beta}T^{+}_{\alpha,\beta})f)(x)dx\geq0.$$
Putting \eqref{strat_TT1_ab} into this inequality, one calculates
\begin{equation}\label{strat_TT_int_part1_ab}
\begin{split}
2\int_{\G}|(\mathcal{L}f)(x)|^{2}dx&+(2\alpha(N-2)-4\beta)\int_{\G}\frac{\overline{f(x)}(\L f)(x)}{|x'|^{2}}dx\\&
+(-4\alpha(N-2)-2\alpha^{2}(N-3)+8\beta)\int_{\G}\frac{\overline{f(x)}(x'\cdot (\nabla_{H}f)(x))}{|x'|^{4}}dx\\
&+(2\alpha(N-2)(4-N)+2\alpha^{2}(N-2)-2\alpha\beta(N-2)\\
&+(4N-16)\beta+2\beta^{2})\int_{\G}\frac{|f(x)|^{2}}{|x'|^{4}}dx\\
&-2\alpha^{2}\sum_{j,k=1}^{N}\int_{\G}\frac{\overline{f(x)}x'_{j}x'_{k}
(X_{j}X_{k}f)(x)}{|x'|^{4}}dx\geq0.
\end{split}
\end{equation}
Using the identities
\begin{equation}\label{strat_TT_int_part2_ab}
\begin{aligned}
\int_{\G}\frac{\overline{f(x)}(\L f)(x)}{|x'|^{2}}dx
& =-\sum_{j=1}^{N}\int_{\G}\frac{\overline{(X_{j}f)(x)}(X_{j}f)(x)}
{|x'|^{2}} \\
& \quad -\sum_{j=1}^{N}\int_{\G}\overline{f(x)}(-2)|x'|^{-3}X_{j}
|x'|(X_{j}f)(x)dx \\
& =2\int_{\G}\frac{\overline{f(x)}(x'\cdot (\nabla_{H}f)(x))}{|x'|^{4}}dx
-\int_{\G}\frac{|(\nabla_{H}f)(x)|^{2}}{|x'|^{2}}dx
\end{aligned}
\end{equation}
and
$$
\sum_{j,k=1}^{N}\int_{\G}\frac{\overline{f(x)}x'_{j}x'_{k}(X_{j}X_{k}f)(x)}{|x'|^{4}}dx$$
$$=-(N-1)\sum_{k=1}^{N}\int_{\G}\frac{\overline{f(x)}x'_{k}(X_{k}f)(x)}{|x'|^{4}}dx
-2\sum_{k=1}^{N}\int_{\G}\frac{\overline{f(x)}x'_{k}(X_{k}f)(x)}{|x'|^{4}}dx$$
$$-\sum_{j,k=1}^{N}\int_{\G}\frac{x'_{j}x'_{k}\overline{(X_{j}f)(x)}
(X_{k}f)(x)}{|x'|^{4}}dx
-\sum_{j,k=1}^{N}\int_{\G}\overline{f(x)}x'_{j}x'_{k}(X_{k}f)(x)(-4)|x'|^{-5}X_{j}|x'|dx$$$$=
-(N+1)\sum_{k=1}^{N}\int_{\G}\frac{\overline{f(x)}x'_{k}(X_{k}f)(x)}{|x'|^{4}}dx
+4\sum_{j,k=1}^{N}\int_{\G}\frac{\overline{f(x)}(x'_{j})^{2}x'_{k}(X_{k}f)(x)}
{|x'|^{6}}dx$$$$
-\int_{\G}\frac{|x'\cdot (\nabla_{H}f)(x)|^{2}}{|x'|^{4}}dx
=-(N-3)\int_{\G}\frac{\overline{f(x)}(x'\cdot (\nabla_{H}f)(x))}{|x'|^{4}}dx$$
\begin{equation}\label{strat_TT_int_part3_ab}
-\int_{\G}\frac{|x'\cdot (\nabla_{H}f)(x)|^{2}}{|x'|^{4}}dx
\end{equation}
in \eqref{strat_TT_int_part1_ab}, we obtain
$$
2\int_{\G}|(\mathcal{L}f)(x)|^{2}dx
+(4\alpha(N-2)-8\beta-4\alpha(N-2)+8\beta-2\alpha^{2}(N-3)+2\alpha^{2}(N-3))$$
$$\times\int_{\G}\frac{\overline{f(x)}(x'\cdot (\nabla_{H}f)(x))}{|x'|^{4}}dx$$
$$+(2\alpha(N-2)(4-N)+2\alpha^{2}(N-2)-2\alpha\beta(N-2)+(4N-16)\beta+2\beta^{2})\int_{\G}\frac{|f(x)|^{2}}{|x'|^{4}}dx$$
$$-(2\alpha(N-2)-4\beta)\int_{\G}\frac{|(\nabla_{H} f)(x)|^{2}}{|x'|^{2}}dx$$$$
+2\alpha^{2}\int_{\G}\frac{|x'\cdot (\nabla_{H}f)(x)|^{2}}{|x'|^{4}}dx\geq0,
$$
which implies \eqref{strat_Rellich1_ab}.
\end{proof}
Now let us give a very elementary proof of a version of the Hardy inequality on stratified groups using the factorization method. We note that this inequality on stratified groups was obtained in \cite{RS17a} by a different method.

\begin{thm}\label{Hardy_strat_thm}
Let $\mathbb{G}$ be a stratified group with $N\geq3$ being the dimension of the first stratum. Let $\alpha\in\mathbb{R}$. Then for all complex-valued functions $f\in C_{0}^{\infty}(\mathbb{G}\backslash\{x'=0\})$ we have
\begin{equation}\label{Hardy_strat}\|\nabla_{H}f\|_{L^{2}(\G)}\geq \frac{N-2}{2}\left\|\frac{f}{|x'|}\right\|_{L^{2}(\G)},
\end{equation}
where the constant $\frac{N-2}{2}$ is sharp.
\end{thm}
\begin{proof}[Proof of Theorem \ref{Hardy_strat_thm}] Here we use the following one-parameter differential expression
$$\widetilde{T}_{\alpha}:=\nabla_{H}+\alpha\frac{x'}{|x'|^{2}},$$
and its formal adjoint
$$\widetilde{T}^{+}_{\alpha}:=-{\rm div}_{H}(\cdot)+\alpha\frac{x'\cdot}{|x'|^{2}},$$
where $x'\neq0$. Taking into account \eqref{formula2} we have
$$\widetilde{T}^{+}_{\alpha}\widetilde{T}_{\alpha}f=-(\L f)(x)-\alpha {\rm div}_{H}\left(\frac{x'}{|x'|^{2}}f\right)(x)+\alpha\frac{x'\cdot(\nabla_{H}f)(x)}{|x'|^{2}}+\alpha^{2}\frac{f(x)}{|x'|^{2}}$$
$$=-(\L f)(x)-\alpha {\rm div}_{H}\left(\frac{x'}{|x'|^{2}}\right)f(x)+\alpha^{2}\frac{f(x)}{|x'|^{2}}$$
$$=-(\L f)(x)+\frac{\alpha(\alpha+2-N)}{|x'|^{2}}f(x).$$
By integrating by parts and using the nonnegativity of $\widetilde{T}^{+}_{\alpha}\widetilde{T}_{\alpha}$ one calculates
\begin{equation*}
\begin{split}0&\leq \int_{\G}|\widetilde{T}_{\alpha}f|^{2}dx=\int_{\G}\overline{f(x)}(\widetilde{T}^{+}_{\alpha}\widetilde{T}_{\alpha}f)(x)dx\\
&=\int_{\G}|\nabla_{H}f|^{2}dx+\alpha(\alpha+2-N)\int_{\G}\frac{|f(x)|^{2}}{|x'|^{2}}dx.
\end{split}
\end{equation*}
It follows that
$$\int_{\G}|\nabla_{H}f|^{2}dx\geq \alpha(N-2-\alpha)\int_{\G}\frac{|f(x)|^{2}}{|x'|^{2}}dx,$$
which after maximising the constant in the above inequality with respect to $\alpha$, we obtain \eqref{Hardy_strat}. The sharpness of the constant in the obtained inequality was shown in \cite{RS17a}.
\end{proof}

The interesting result here is that by modifying the differential expression $\widetilde{T}_{\alpha}$, the factorization method gives a refinement of the Hardy inequality
\eqref{Hardy_strat}:

\begin{thm}\label{impr_Hardy_strat_thm}
Let $\mathbb{G}$ be a stratified group with $N\geq3$ being the dimension of the first stratum. Let $\alpha\in\mathbb{R}$. Then for all complex-valued functions $f\in C_{0}^{\infty}(\mathbb{G}\backslash\{x'=0\})$ we have
\begin{equation}\label{impr_Hardy_strat}\left\|\frac{x'\cdot\nabla_{H}f}{|x'|}\right\|_{L^{2}(\G)}\geq \frac{N-2}{2}\left\|\frac{f}{|x'|}\right\|_{L^{2}(\G)},
\end{equation}
where the constant $\frac{N-2}{2}$ is sharp.
\end{thm}

\begin{rem} We note that the estimate \eqref{impr_Hardy_strat} implies \eqref{Hardy_strat} by the Cauchy-Schwarz inequality. Consequently, the sharpness of the constant in \eqref{impr_Hardy_strat} follows from the sharpness of the constant in \eqref{Hardy_strat}.
\end{rem}

\begin{proof}[Proof of Theorem \ref{impr_Hardy_strat_thm}] Here we take the one parameter differential expressions in the form
$$\widehat{T}_{\alpha}:=\frac{x'\cdot\nabla_{H}}{|x'|}+\frac{\alpha}{|x'|},$$
and
$$\widehat{T}_{\alpha}^{+}:=-\frac{x'\cdot\nabla_{H}}{|x'|}+\frac{\alpha-N+1}{|x'|},$$
where $x'\neq0$. By a direct calculation and using \eqref{formula1} we get
\begin{equation*}
\begin{aligned}
& \left(\frac{x'\cdot\nabla_{H}}{|x'|}\right)
\left(\frac{x'\cdot(\nabla_{H}f)(x)}{|x'|}\right) \\
& =\frac{\sum_{j,k=1}^{N}(x'_{j}X_{j})(x'_{k}(X_{k}f)(x))}{|x'|^{2}}+
\sum_{j,k=1}^{N}x'_{j}(-1)|x'|^{-2}X_{j}|x'|\frac{x'_{k}(X_{k}f)(x)}{|x'|} \\
& =\frac{\sum_{k=1}^{N}x'_{k}(X_{k}f)(x)}{|x'|^{2}}+\frac{\sum_{j,k=1}^{N}x'_{j}x'_{k}(X_{j}X_{k}f)(x)}{|x'|^{2}}
-\frac{\sum_{k=1}^{N}x'_{k}(X_{k}f)(x)}{|x'|^{2}} \\
& =\frac{\sum_{j,k=1}^{N}x'_{j}x'_{k}(X_{j}X_{k}f)(x)}{|x'|^{2}}
\end{aligned}
\end{equation*}
and
\begin{equation*}
\begin{aligned}
(x'\cdot\nabla_{H})\left(\frac{f}{|x'|}\right) & =\frac{\sum_{k=1}^{N}x'_{k}(X_{k}f)(x)}{|x'|}+
\sum_{k=1}^{N}x'_{k}(-1)|x'|^{-2}X_{k}|x'|f(x) \\
& =\frac{x'\cdot(\nabla_{H}f)(x)}{|x'|}-\frac{f(x)}{|x'|}.
\end{aligned}
\end{equation*}
Taking into account these identities, we obtain
\begin{multline*}
\widehat{T}_{\alpha}^{+}\widehat{T}_{\alpha}f(x) \\
=-\frac{\sum_{j,k=1}^{N}x'_{j}x'_{k}(X_{j}X_{k}f)(x)}{|x'|^{2}}-
(N-1)\frac{x'\cdot(\nabla_{H}f)(x)}{|x'|^{2}}+\frac{\alpha(\alpha+2-N)}{|x'|^{2}}f(x).
\end{multline*}
Using \eqref{formula1}, we calculate
$$
\sum_{j,k=1}^{N}\int_{\G}\frac{\overline{f(x)}x'_{j}x'_{k}(X_{j}X_{k}f)(x)}{|x'|^{2}}dx$$
$$=-(N-1)\sum_{k=1}^{N}\int_{\G}\frac{\overline{f(x)}x'_{k}(X_{k}f)(x)}{|x'|^{2}}dx
-2\sum_{k=1}^{N}\int_{\G}\frac{\overline{f(x)}x'_{k}(X_{k}f)(x)}{|x'|^{2}}dx$$
$$-\sum_{j,k=1}^{N}\int_{\G}\frac{x'_{j}x'_{k}\overline{(X_{j}f)(x)}(X_{k}f)(x)}{|x'|^{2}}dx
-\sum_{j,k=1}^{N}\int_{\G}\overline{f(x)}x'_{j}x'_{k}(X_{k}f)(x)(-2)|x'|^{-3}X_{j}|x'|dx$$$$=
-(N+1)\sum_{k=1}^{N}\int_{\G}\frac{\overline{f(x)}x'_{k}(X_{k}f)(x)}{|x'|^{2}}dx
+2\sum_{j,k=1}^{N}\int_{\G}\frac{\overline{f(x)}(x'_{j})^{2}x'_{k}(X_{k}f)(x)}{|x'|^{4}}dx$$$$
-\int_{\G}\frac{|x'\cdot (\nabla_{H}f)(x)|^{2}}{|x'|^{2}}dx
=-(N-1)\int_{\G}\frac{\overline{f(x)}(x'\cdot (\nabla_{H}f)(x))}{|x'|^{2}}dx$$
\begin{equation}\label{case3}
-\int_{\G}\frac{|x'\cdot (\nabla_{H}f)(x)|^{2}}{|x'|^{2}}dx.
\end{equation}
Taking into account this, integrating by parts, and using the nonnegativity of the operator $\widehat{T}_{\alpha}^{+}\widehat{T}_{\alpha}$, we get
\begin{equation*}
\begin{split}0&\leq \int_{\G}|\widehat{T}_{\alpha}f|^{2}dx=
\int_{\G}\overline{f(x)}(\widehat{T}_{\alpha}^{+}\widehat{T}_{\alpha}f)(x)dx\\
&=-\int_{\G}\left(\frac{\sum_{j,k=1}^{N}x'_{j}x'_{k}\overline{f(x)}(X_{j}X_{k}f)(x)}{|x'|^{2}}+
\frac{(N-1)\overline{f(x)}(x'\cdot(\nabla_{H}f)(x))}{|x'|^{2}}
\right)dx
\end{split}
\end{equation*}
$$+\alpha(\alpha-N+2)\int_{\G}\frac{|f(x)|^{2}}{|x'|^{2}}dx.$$
Consequently, using \eqref{case3} one obtains
$$\int_{\G}\left(\frac{|x'\cdot(\nabla_{H}f)(x)|^{2}}{|x'|^{2}}+\alpha(\alpha-N+2)\frac{|f(x)|^{2}}{|x'|^{2}}\right)dx\geq0.$$
It now follows that
$$\int_{\G}\frac{|x'\cdot(\nabla_{H}f)(x)|^{2}}{|x'|^{2}}dx\geq \alpha((N-2)-\alpha)\int_{\G}\frac{|f(x)|^{2}}{|x'|^{2}}dx.$$
As usual, by maximising $\alpha((N-2)-\alpha)$ with respect to $\alpha$, we obtain \eqref{impr_Hardy_strat}.
\end{proof}
\section{Factorizations on the Heisenberg group}
\label{SEC:Fac_Heisenberg}

Now we discuss the Hardy-Rellich inequalities on the Heisenberg group. In this section we adopt all the notation concerning the Heisenberg group from the introduction as well as some new notation that will be useful in the computations:
\begin{equation}\label{heisen_X}
\widetilde{x}'_{j}=
\begin{cases} x_{j}, 1\leq j\leq n;\\
y_{j-n}, n+1\leq j\leq 2n,   \end{cases},
\;\;\;\widetilde{X}_{j}=
\begin{cases} X_{j}, 1\leq j\leq n;\\
Y_{j-n}, n+1\leq j\leq 2n,   \end{cases}
\end{equation}
and
\begin{equation}\label{heisen_norm}
|\widetilde{x}'|=\sqrt{\sum_{j=1}^{N}({\widetilde{x}'}_{j})^{2}}=\sqrt{\sum_{j=1}^{n}(x^{2}_{j}+y^{2}_{j})}.
\end{equation}
Taking into account the above notations, $\nabla_H$ is defined by
\begin{equation}\label{heisen_grad}\nabla_{H}:=(\widetilde{X}_{1}, \ldots, \widetilde{X}_{N}),
\end{equation}
and the sub-Laplacian
\begin{equation}\label{heisen_L}\L:=\sum_{j=1}^{N}\widetilde{X}^{2}_{j}
=\sum_{j=1}^{n}
\left(\partial_{\widetilde{x}'_{j}}-\frac{\widetilde{x}'_{j+n}}{2}\partial_{t}\right)^{2}+\left(\partial_{\widetilde{x}'_{j+n}}+
\frac{\widetilde{x}'_{j}}{2}\partial_{t}\right)^{2}.
\end{equation}
Therefore, we see that $N=2n$ and the formulae \eqref{formula1} and \eqref{formula2} also hold on the Heisenberg group.

For the following formulation we recall the tangential derivative operator
$Z=\sum_{j=1}^n (x_j\partial_{y_j}-y_j\partial_{x_j})$ given in \eqref{EQ:Lap2}.

\begin{thm}\label{heisen_Rellich} Let $\alpha, \beta\in \mathbb{R}$. Then for all complex-valued functions $f\in C_{0}^{\infty}(\mathbb{H}_{n}\backslash\{\widetilde{x}'=0\})$ we have
\begin{equation}\label{heisen_Rellich1}
\begin{split}
\|\mathcal{L} f\|^{2}_{L^{2}(\mathbb{H}_{n})}&\geq ((N-4)\alpha-2\beta)\left\|\frac{\nabla_{H} f}{|\widetilde{x}'|}\right\|^{2}_{L^{2}(\mathbb{H}_{n})}\\
&-\alpha(\alpha-4)\left\|\frac{\widetilde{x}'\cdot\nabla_{H} f}{|\widetilde{x}'|^{2}}\right\|^{2}_{L^{2}(\mathbb{H}_{n})}\\
&+\beta((N-4)(\alpha-2)-\beta)\left\|\frac{f}{|\widetilde{x}'|^{2}}\right\|^{2}_{L^{2}(\mathbb{H}_{n})}\\
&- 2\alpha
 \int_{\mathbb{H}_{n}} \frac{\overline{f(\widetilde{x}',t)} Z Tf(\widetilde{x}',t)}{|\widetilde{x}'|^{2}}
d\widetilde{x}' dt
+ \alpha \|Tf\|^2_{L^{2}(\mathbb{H}_{n})}.
\end{split}
\end{equation}
and
\begin{equation}\label{heisen_Rellich11}
\begin{split}
\|\mathcal{L} f\|^{2}_{L^{2}(\mathbb{H}_{n})}&\geq (N\alpha-2\beta)\left\|\frac{\nabla_{H} f}{|\widetilde{x}'|}\right\|^{2}_{L^{2}(\mathbb{H}_{n})}\\
&-\alpha(\alpha+4)\left\|\frac{\widetilde{x}'\cdot\nabla_{H} f}{|\widetilde{x}'|^{2}}\right\|^{2}_{L^{2}(\mathbb{H}_{n})}\\
&+(2(N-4)(\alpha(N-2)-\beta)-2\alpha^{2}(N-2)+\alpha\beta N-\beta^{2})\left\|\frac{f}{|\widetilde{x}'|^{2}}\right\|^{2}_{L^{2}(\mathbb{H}_{n})}\\
&+ 2\alpha
 \int_{\mathbb{H}_{n}} \frac{\overline{f(\widetilde{x}',t)} Z Tf(\widetilde{x}',t)}{|\widetilde{x}'|^{2}}
d\widetilde{x}' dt
- \alpha \|Tf\|^2_{L^{2}(\mathbb{H}_{n})}.
\end{split}
\end{equation}
Moreover,
\begin{equation}\label{heisen_Rellich2}
\begin{split}
\|\mathcal{L} f\|^{2}_{L^{2}(\mathbb{H}_{n})}&\geq ((N-\alpha)\alpha-2\beta)\left\|\frac{\nabla_{H} f}{|\widetilde{x}'|}\right\|^{2}_{L^{2}(\mathbb{H}_{n})}\\
&+\beta((N-4)(\alpha-2)-\beta)\left\|\frac{f}{|\widetilde{x}'|^{2}}\right\|^{2}_{L^{2}(\mathbb{H}_{n})}\\
&- 2\alpha
 \int_{\mathbb{H}_{n}} \frac{\overline{f(\widetilde{x}',t)} Z Tf(\widetilde{x}',t)}{|\widetilde{x}'|^{2}}
d\widetilde{x}' dt
+ \alpha \|Tf\|^2_{L^{2}(\mathbb{H}_{n})}
\end{split}
\end{equation}
for $\alpha(\alpha-4)\geq0$, and
\begin{equation}\label{heisen_Rellich22}
\begin{split}
\|\mathcal{L} f\|^{2}_{L^{2}(\mathbb{H}_{n})}&\geq (-2\beta+\alpha(N-\alpha)-4\alpha)\left\|\frac{\nabla_{H} f}{|\widetilde{x}'|}\right\|^{2}_{L^{2}(\mathbb{H}_{n})}\\
&+(2(N-4)(\alpha(N-2)-\beta)-2\alpha^{2}(N-2)+\alpha\beta N-\beta^{2})\left\|\frac{f}{|\widetilde{x}'|^{2}}\right\|^{2}_{L^{2}(\mathbb{H}_{n})}\\
&+ 2\alpha
 \int_{\mathbb{H}_{n}} \frac{\overline{f(\widetilde{x}',t)} Z Tf(\widetilde{x}',t)}{|\widetilde{x}'|^{2}}
d\widetilde{x}' dt
- \alpha \|Tf\|^2_{L^{2}(\mathbb{H}_{n})}
\end{split}
\end{equation}
for $\alpha(\alpha+4)\geq0$.
\end{thm}
\begin{cor}\label{cor_Rn2} In the abelian case $\G =(\Rn, +)$, we have $\nabla_{H}=\nabla=(\partial_{x_{1}},...,\partial_{x_{n}})$ and we can note that if we argue as in the proof of Theorem \ref{heisen_Rellich}, the last two terms in \eqref{heisen_Rellich1} and \eqref{heisen_Rellich11}  vanish. Therefore, in this case, \eqref{heisen_Rellich1} would coincide with the Gesztesy and Littlejohn's estimate \eqref{Gesztesy_ineq}, while \eqref{heisen_Rellich11} would give the following new estimate for $\alpha,\beta\in \mathbb{R}$ and $f\in C_{0}^{\infty}(\Rn\backslash\{0\})$ with $n\geq2$:
\begin{equation}\label{heisen_Rellich11_Rn}
\begin{split}
& \|\triangle f\|^{2}_{L^{2}(\Rn)}\\
&\geq (n\alpha-2\beta)\left\|\frac{\nabla f}{|x|}\right\|^{2}_{L^{2}(\Rn)}\\
&-\alpha(\alpha+4)\left\|\frac{x\cdot\nabla f}{|x|^{2}}\right\|^{2}_{L^{2}(\Rn)}\\
&+(2(n-4)(\alpha(n-2)-\beta)-2\alpha^{2}(n-2)+\alpha\beta n-\beta^{2})\left\|\frac{f}{|x|^{2}}\right\|^{2}_{L^{2}(\Rn)}.
\end{split}
\end{equation}
\end{cor}
\begin{rem}\label{REM:real} Let us show that the following summand in the Theorem \ref{heisen_Rellich} is actually real-valued. Indeed, using integration by parts, we calculate
$$\int_{\mathbb{H}_{n}} \frac{\overline{f(\widetilde{x}',t)} Z Tf(\widetilde{x}',t)}{|\widetilde{x}'|^{2}}
d\widetilde{x}' dt=\sum_{j=1}^n\int_{\mathbb{H}_{n}}
\frac{\overline{f(\widetilde{x}',t)}(x_j\partial_{y_j}-y_j\partial_{x_j})(\partial_{t}f(\widetilde{x}',t))}{|\widetilde{x}'|^{2}}
d\widetilde{x}' dt$$
$$=-\frac{1}{2}\sum_{j=1}^n\int_{\mathbb{H}_{n}}\overline{\partial_{t}f(\widetilde{x}',t)}\left(\partial_{y_j}(f(\widetilde{x}',t))
\partial_{x_j}(\log|\widetilde{x}'|^{2})-\partial_{x_j}(f(\widetilde{x}',t))
\partial_{y_j}(\log|\widetilde{x}'|^{2})\right)d\widetilde{x}' dt
$$
$$=\frac{1}{2}\sum_{j=1}^n\int_{\mathbb{H}_{n}}\log|\widetilde{x}'|^{2}
(\partial_{y_j}(f(\widetilde{x}',t))\overline{\partial_{tx_j}f(\widetilde{x}',t)}-
\partial_{x_j}(f(\widetilde{x}',t))\overline{\partial_{ty_j}f(\widetilde{x}',t)})d\widetilde{x}' dt$$
$$=\frac{1}{2}\sum_{j=1}^n\int_{\mathbb{H}_{n}}\log|\widetilde{x}'|^{2}
(\partial_{y_j}(f(\widetilde{x}',t))\overline{\partial_{tx_j}f(\widetilde{x}',t)}
+\overline{\partial_{y_j}(f(\widetilde{x}',t))\overline{\partial_{tx_j}f(\widetilde{x}',t)}}
)d\widetilde{x}' dt$$
$$
=\sum_{j=1}^n\int_{\mathbb{H}_{n}}\log|\widetilde{x}'|^{2}
{\rm Re}(\partial_{y_j}(f(\widetilde{x}',t))\overline{\partial_{tx_j}f(\widetilde{x}',t)}
)d\widetilde{x}' dt.
$$
\end{rem}
\begin{rem}
We note that we can use that $ZT=TZ$ and integrate by parts in $t$, so that the term
$$
 \int_{\mathbb{H}_{n}} \frac{\overline{f(\widetilde{x}',t)} Z Tf(\widetilde{x}',t)}{|\widetilde{x}'|^{2}}
d\widetilde{x}' dt = - \left(\frac{Zf}{|\widetilde{x}'|}, \frac{Tf}{|\widetilde{x}'|}\right)_{L^{2}(\mathbb{H}_{n})}
$$
can be seen as an interaction between the central and the tangential derivatives.
This formula also implies \eqref{EQ:Heis1}-\eqref{EQ:Heis22} using \eqref{heisen_Rellich1}-\eqref{heisen_Rellich22}, respectively.

Writing \eqref{EQ:Lap2} in the form
\begin{equation}\label{EQ:Lap3}
ZT=\L-\Delta_{\widetilde{x}'}-\frac{|\widetilde{x}'|^2}{4}T^2,
\end{equation}
we have
\begin{multline}\label{EQ:Lapfinal}
\int_{\mathbb{H}_{n}} \frac{\overline{f(x)} Z  Tf(x)}{|\widetilde{x}'|^{2}}
d\widetilde{x}' dt \\
=\left(\frac{\L f}{|\widetilde{x}'|},\frac{f}{|\widetilde{x}'|}\right)_{L^2(\mathbb{H}_{n})}
-\left(\frac{\Delta_{\widetilde{x}'} f}{|\widetilde{x}'|},\frac{f}{|\widetilde{x}'|}\right)_{L^2(\mathbb{H}_{n})}+\frac14 \|Tf\|^2_{L^2(\mathbb{H}_{n})}.
\end{multline}
Consequently, inequalities \eqref{heisen_Rellich1} and \eqref{heisen_Rellich11} can be also written as
\begin{equation}\label{heisen_Rellich1-1}
\begin{split}
\|\mathcal{L} f\|^{2}_{L^{2}(\mathbb{H}_{n})}&\geq ((N-4)\alpha-2\beta)\left\|\frac{\nabla_{H} f}{|\widetilde{x}'|}\right\|^{2}_{L^{2}(\mathbb{H}_{n})}\\
&-\alpha(\alpha-4)\left\|\frac{\widetilde{x}'\cdot\nabla_{H} f}{|\widetilde{x}'|^{2}}\right\|^{2}_{L^{2}(\mathbb{H}_{n})}\\
&+\beta((N-4)(\alpha-2)-\beta)\left\|\frac{f}{|\widetilde{x}'|^{2}}\right\|^{2}_{L^{2}(\mathbb{H}_{n})}\\
&- 2\alpha \left(\frac{\L f}{|\widetilde{x}'|},\frac{f}{|\widetilde{x}'|}\right)_{L^2(\mathbb{H}_{n})}
+2\alpha\left(\frac{\Delta_{\widetilde{x}'} f}{|\widetilde{x}'|},\frac{f}{|\widetilde{x}'|}\right)_{L^2(\mathbb{H}_{n})}
+ \frac{\alpha}{2} \|Tf\|^{2}_{L^{2}(\mathbb{H}_{n})},
\end{split}
\end{equation}
and
\begin{equation}\label{heisen_Rellich1-11}
\begin{split}
\|\mathcal{L} f\|^{2}_{L^{2}(\mathbb{H}_{n})}&\geq (N\alpha-2\beta)\left\|\frac{\nabla_{H} f}{|\widetilde{x}'|}\right\|^{2}_{L^{2}(\mathbb{H}_{n})}\\
&-\alpha(\alpha+4)\left\|\frac{\widetilde{x}'\cdot\nabla_{H} f}{|\widetilde{x}'|^{2}}\right\|^{2}_{L^{2}(\mathbb{H}_{n})}\\
&+(2(N-4)(\alpha(N-2)-\beta)-2\alpha^{2}(N-2)+\alpha\beta N-\beta^{2})\left\|\frac{f}{|\widetilde{x}'|^{2}}\right\|^{2}_{L^{2}(\mathbb{H}_{n})}\\
&+2\alpha \left(\frac{\L f}{|\widetilde{x}'|},\frac{f}{|\widetilde{x}'|}\right)_{L^2(\mathbb{H}_{n})}
-2\alpha\left(\frac{\Delta_{\widetilde{x}'} f}{|\widetilde{x}'|},\frac{f}{|\widetilde{x}'|}\right)_{L^2(\mathbb{H}_{n})}
-\frac{\alpha}{2} \|Tf\|^{2}_{L^{2}(\mathbb{H}_{n})},
\end{split}
\end{equation}
respectively.
\end{rem}

\begin{proof}[Proof of Theorem \ref{heisen_Rellich}]
We will be applying the factorization method with the following differential expressions for two real parameters $\alpha,\beta\in\mathbb R$,
\begin{equation}\label{dif_exp1}
T_{\alpha, \beta}:=-\mathcal{L}+\alpha\frac{\widetilde{x}'\cdot\nabla_{H}}{|\widetilde{x}'|^{2}}+\frac{\beta}{|\widetilde{x}'|^{2}}
\end{equation}
and its formal adjoint
\begin{equation}\label{dif_exp2}T^{+}_{\alpha,\beta}:=-\mathcal{L}-\alpha\frac{\widetilde{x}'\cdot\nabla_{H}}{|\widetilde{x}'|^{2}}
-\frac{\alpha(N-2)-\beta}{|\widetilde{x}'|^{2}},
\end{equation}
where $\widetilde{x}'\neq0$. Then, by a direct calculation for any function $f\in C_{0}^{\infty}(\mathbb{H}_{n}\backslash\{\widetilde{x}'=0\})$ we have
\begin{equation}\label{heisen_TT0}
\begin{aligned}
& (T^{+}_{\alpha,\beta}T_{\alpha,\beta}f)(x) \\
& =\left(-\mathcal{L}-\alpha\frac{\widetilde{x}'\cdot\nabla_{H}}{|\widetilde{x}'|^{2}}
-\frac{\alpha(N-2)-\beta}{|\widetilde{x}'|^{2}}\right)\left(-(\mathcal{L}f)(x)
+\alpha\frac{\widetilde{x}'\cdot(\nabla_{H}f)}{|\widetilde{x}'|^{2}}+
\frac{\beta f(x)}{|\widetilde{x}'|^{2}}\right) \\
& =(\L^{2}f)(x)+\alpha\left(-\L\left(\frac{\widetilde{x}'\cdot(\nabla_{H}f)}{|\widetilde{x}'|^{2}}\right)(x)+\frac{\widetilde{x}'\cdot\nabla_{H}}{|\widetilde{x}'|^{2}}(\L f)(x)+
\frac{N-2}{|\widetilde{x}'|^{2}}(\L f)(x)\right) \\
& \quad +\beta\left(-\L\left(\frac{f}{|\widetilde{x}'|^{2}}\right)(x)-\frac{(\L f)(x)}{|\widetilde{x}'|^{2}}\right) \\
& \quad +\alpha\beta\left(-\frac{\widetilde{x}'\cdot\nabla_{H}}{|\widetilde{x}'|^{2}}\left(\frac{f}{|\widetilde{x}'|^{2}}\right)(x)+
\frac{\widetilde{x}'\cdot(\nabla_{H}f)(x)}{|\widetilde{x}'|^{4}}-\frac{(N-2)f(x)}{|\widetilde{x}'|^{4}}\right)
\\
& \; +\alpha^{2}\left(-\left(\frac{\widetilde{x}'\cdot\nabla_{H}}{|\widetilde{x}'|^{2}}\right)
\left(\frac{\widetilde{x}'\cdot(\nabla_{H}f)}{|\widetilde{x}'|^{2}}\right)(x)-
(N-2)\frac{\widetilde{x}'\cdot(\nabla_{H}f)(x)}{|\widetilde{x}'|^{4}}\right)+\beta^{2}\frac{f(x)}{|\widetilde{x}'|^{4}} \\
& =:(\L^{2}f)(x)+\alpha J_{1}+\beta J_{2}+\alpha\beta J_{3}+\alpha^{2}J_{4}+\beta^{2}J_{5}.
\end{aligned}
\end{equation}
In order to simplify this, let us first calculate $J_{1}$:
$$J_{1}=-\L\left(\frac{\widetilde{x}'\cdot(\nabla_{H}f)}{|\widetilde{x}'|^{2}}\right)(x)
+\frac{\widetilde{x}'\cdot\nabla_{H}}{|\widetilde{x}'|^{2}}(\L f)(x)+
\frac{N-2}{|\widetilde{x}'|^{2}}(\L f)(x)$$
$$=-\sum_{j=1}^{N}\widetilde{X}^{2}_{j}
\left(\frac{\sum_{k=1}^{N}\widetilde{x}'_{k}(\widetilde{X}_{k}f)}{|\widetilde{x}'|^{2}}\right)(x)
+\frac{\widetilde{x}'\cdot\nabla_{H}}{|\widetilde{x}'|^{2}}(\L f)(x)+
\frac{N-2}{|\widetilde{x}'|^{2}}(\L f)(x)$$
$$=-\sum_{j=1}^{N}\widetilde{X}_{j}
\left(-\frac{2\widetilde{x}'_{j}}{|\widetilde{x}'|^{4}}\sum_{k=1}^{N}\widetilde{x}'_{k}(\widetilde{X}_{k}f)(x)+
\frac{(\widetilde{X}_{j}f)+\sum_{k=1}^{N}\widetilde{x}'_{k}(\widetilde{X}_{j}\widetilde{X}_{k}f)}{|\widetilde{x}'|^{2}}
\right)(x)$$
$$+\frac{\widetilde{x}'\cdot\nabla_{H}}{|\widetilde{x}'|^{2}}(\L f)(x)+
\frac{N-2}{|\widetilde{x}'|^{2}}(\L f)(x)$$
$$=\sum_{j=1}^{N}\left(-\frac{8(\widetilde{x}'_{j})^{2}}{|\widetilde{x}'|^{6}}\sum_{k=1}^{N}\widetilde{x}'_{k}(\widetilde{X}_{k}f)(x)
+\frac{2}{|\widetilde{x}'|^{4}}\sum_{k=1}^{N}\widetilde{x}'_{k}(\widetilde{X}_{k}f)(x)\right)$$
$$+\sum_{j=1}^{N}
\frac{2\widetilde{x}'_{j}}{|\widetilde{x}'|^{4}}\left((\widetilde{X}_{j}f)(x)+\sum_{k=1}^{N}\widetilde{x}'_{k}(\widetilde{X}_{j}\widetilde{X}_{k}f)(x)\right)$$
$$+\sum_{j=1}^{N}\frac{2\widetilde{x}'_{j}\left((\widetilde{X}_{j}f)(x)+\sum_{k=1}^{N}\widetilde{x}'_{k}(\widetilde{X}_{j}\widetilde{X}_{k}f)(x)\right)}{|\widetilde{x}'|^{4}}$$
$$-\sum_{j=1}^{N}\frac{(2\widetilde{X}^{2}_{j}f)(x)+\sum_{k=1}^{N}\widetilde{x}'_{k}(\widetilde{X}^{2}_{j}\widetilde{X}_{k}f)(x)}{|\widetilde{x}'|^{2}}$$
$$+\frac{\widetilde{x}'\cdot\nabla_{H}}{|\widetilde{x}'|^{2}}(\L f)(x)+
\frac{N-2}{|\widetilde{x}'|^{2}}(\L f)(x)$$
$$=\frac{2N-4}{|\widetilde{x}'|^{4}}\widetilde{x}'\cdot(\nabla_{H}f)(x)+\frac{4}{|\widetilde{x}'|^{4}}\sum_{j,k=1}^{N}\widetilde{x}'_{j}\widetilde{x}'_{k}(\widetilde{X}_{j}\widetilde{X}_{k}f)(x)+
\frac{N-4}{|\widetilde{x}'|^{2}}(\L f)(x) $$
\begin{equation}\label{heisen_TT01}
-\frac{\sum_{k=1}^{N}\widetilde{x}'_{k}\L (\widetilde{X}_{k}f)}{|\widetilde{x}'|^{2}}(x)+\frac{\widetilde{x}'\cdot\nabla_{H}}{|\widetilde{x}'|^{2}}(\L f)(x).
\end{equation}
Here using two times the formulae
$$[\widetilde{X}_{j},\widetilde{X}_{k}]=T, \; [\widetilde{X}_{k},\widetilde{X}_{j}]=-T, \;\;k=n+j, \; j=1,\ldots,n,$$
for the last two summands in the above equality and noticing that all other commutators
are zero,
distinguishing between the cases $k=j+n$ and $j=k+n$ in the sums below we get
$$-\frac{\sum_{k=1}^{N}\widetilde{x}'_{k}\L (\widetilde{X}_{k}f)}{|\widetilde{x}'|^{2}}(x)+\frac{\widetilde{x}'\cdot\nabla_{H}}{|\widetilde{x}'|^{2}}(\L f)(x)$$
$$=\sum_{j,k=1}^{N}\frac{\widetilde{x}'_{k}}{|\widetilde{x}'|^{2}}((\widetilde{X}_{k}\widetilde{X}_{j}\widetilde{X}_{j}f)(x)-(\widetilde{X}_{j}\widetilde{X}_{j}\widetilde{X}_{k}f)(x))$$
$$=\sum_{j=1}^{n}\frac{\widetilde{x}'_{j+n}}{|\widetilde{x}'|^{2}}((\widetilde{X}_{j+n}\widetilde{X}_{j}\widetilde{X}_{j}f)(x)-(\widetilde{X}_{j}Tf)(x)-(\widetilde{X}_{j}\widetilde{X}_{j+n}\widetilde{X}_{j}f)(x))$$
$$+\sum_{k=1}^{n}\frac{\widetilde{x}'_{k}}{|\widetilde{x}'|^{2}}((\widetilde{X}_{k}\widetilde{X}_{k+n}\widetilde{X}_{k+n}f)(x)+(\widetilde{X}_{k+n}Tf)(x)-(\widetilde{X}_{k+n}\widetilde{X}_{k}\widetilde{X}_{k+n}f)(x))$$
$$=\sum_{j=1}^{n}\frac{\widetilde{x}'_{j+n}}{|\widetilde{x}'|^{2}}((\widetilde{X}_{j+n}\widetilde{X}_{j}\widetilde{X}_{j}f)(x)-2(\widetilde{X}_{j}Tf)(x)-(\widetilde{X}_{j+n}\widetilde{X}_{j}\widetilde{X}_{j}f)(x))$$
$$+\sum_{k=1}^{n}\frac{\widetilde{x}'_{k}}{|\widetilde{x}'|^{2}}((\widetilde{X}_{k}\widetilde{X}_{k+n}\widetilde{X}_{k+n}f)(x)+2(\widetilde{X}_{k+n}Tf)(x)-(\widetilde{X}_{k}\widetilde{X}_{k+n}\widetilde{X}_{k+n}f)(x))$$
$$=2\sum_{j=1}^{n}\frac{({x}_{j}Y_j-y_j X_j)  Tf(x)}{|\widetilde{x}'|^{2}},$$
where in the last line we use the usual notation \eqref{EQ:Heis-nots} for variables and vector fields on the Heisenberg group.
Putting this in \eqref{heisen_TT01}, we obtain
$$
J_{1}=\frac{2N-4}{|\widetilde{x}'|^{4}}\widetilde{x}'\cdot(\nabla_{H}f)(x)+\frac{4}{|\widetilde{x}'|^{4}}\sum_{j,k=1}^{N}\widetilde{x}'_{j}\widetilde{x}'_{k}(\widetilde{X}_{j}\widetilde{X}_{k}f)(x)$$
\begin{equation}\label{J1}+
\frac{N-4}{|\widetilde{x}'|^{2}}(\L f)(x)+2\sum_{j=1}^{n}\frac{({x}_{j}Y_j-y_j X_j)  Tf(x)}{|\widetilde{x}'|^{2}}.
\end{equation}
Now for $J_{2}$ we have
$$J_{2}=-\L\left(\frac{f}{|\widetilde{x}'|^{2}}\right)(x)-\frac{(\L f)(x)}{|\widetilde{x}'|^{2}}=-
\sum_{j=1}^{N}X^{2}_{j}\left(\frac{f}{|\widetilde{x}'|^{2}}\right)(x)-\frac{(\L f)(x)}{|\widetilde{x}'|^{2}}$$
$$=-
\sum_{j=1}^{N}\widetilde{X}_{j}\left(\frac{\widetilde{X}_{j}f}{|\widetilde{x}'|^{2}}
-\frac{2\widetilde{x}'_{j}f}{|\widetilde{x}'|^{4}}\right)(x)-\frac{(\L f)(x)}{|\widetilde{x}'|^{2}}$$
$$=-\sum_{j=1}^{N}\left(\frac{(\widetilde{X}^{2}_{j}f)(x)}{|\widetilde{x}'|^{2}}-
\frac{4\widetilde{x}'_{j}(\widetilde{X}_{j}f)(x)}{|\widetilde{x}'|^{4}}
+\frac{8(\widetilde{x}'_{j})^{2}f(x)}{|\widetilde{x}'|^{6}}
-\frac{2f(x)}{|\widetilde{x}'|^{4}}\right)-\frac{(\L f)(x)}{|\widetilde{x}'|^{2}}$$
\begin{equation}\label{J2}
=-\frac{2(\L f)(x)}{|\widetilde{x}'|^{2}}+\frac{4\widetilde{x}'\cdot(\nabla_{H}f)(x)}{|\widetilde{x}'|^{4}}+(2N-8)\frac{f(x)}{|\widetilde{x}'|^{4}}.
\end{equation}
For $J_{3}$, one calculates
$$J_{3}=-\frac{\widetilde{x}'\cdot\nabla_{H}}{|\widetilde{x}'|^{2}}\left(\frac{f}{|\widetilde{x}'|^{2}}\right)(x)+
\frac{\widetilde{x}'\cdot(\nabla_{H}f)(x)}{|\widetilde{x}'|^{4}}-\frac{(N-2)f(x)}{|\widetilde{x}'|^{4}}$$
$$=-\frac{f(x)}{|\widetilde{x}'|^{2}}(\widetilde{x}'\cdot\nabla_{H})(|\widetilde{x}'|^{-2})-\frac{(N-2)f(x)}{|\widetilde{x}'|^{4}}
=-\frac{f(x)}{|\widetilde{x}'|^{2}}\sum_{j=1}^{N}\widetilde{x}'_{j}\widetilde{X}_{j}(|\widetilde{x}'|^{-2})
-\frac{(N-2)f(x)}{|\widetilde{x}'|^{4}}$$
\begin{equation}\label{J3}
=2\sum_{j=1}^{N}\frac{(\widetilde{x}'_{j})^{2}f(x)}{|\widetilde{x}'|^{6}}-\frac{(N-2)f(x)}{|\widetilde{x}'|^{4}}=-\frac{(N-4)f(x)}{|\widetilde{x}'|^{4}}.
\end{equation}
By a direct calculation, we have for $J_{4}$
\begin{equation}\label{J4}
\begin{aligned}
J_{4} & =-\left(\frac{\widetilde{x}'\cdot\nabla_{H}}{|\widetilde{x}'|^{2}}\right)
\left(\frac{\widetilde{x}'\cdot(\nabla_{H}f)}{|\widetilde{x}'|^{2}}\right)(x)-
(N-2)\frac{\widetilde{x}'\cdot(\nabla_{H}f)(x)}{|\widetilde{x}'|^{4}} \\
& =-\frac{\sum_{j,k=1}^{N}(\widetilde{x}'_{j}\widetilde{X}_{j})
(\widetilde{x}'_{k}(\widetilde{X}_{k}f))(x)}{|\widetilde{x}'|^{4}}-
\sum_{j,k=1}^{N}\widetilde{x}'_{j}(-2)|\widetilde{x}'|^{-3}\widetilde{X}_{j}|\widetilde{x}'|\frac{\widetilde{x}'_{k}(\widetilde{X}_{k}f)(x)}{|\widetilde{x}'|^{2}} \\
& \quad -
(N-2)\frac{\widetilde{x}'\cdot(\nabla_{H}f)(x)}{|\widetilde{x}'|^{4}} \\
& =-\frac{\sum_{k=1}^{N}\widetilde{x}'_{k}(\widetilde{X}_{k}f)(x)}{|\widetilde{x}'|^{4}}-\frac{\sum_{j,k=1}^{N}\widetilde{x}'_{j}\widetilde{x}'_{k}\widetilde{X}_{j}(\widetilde{X}_{k}f)(x)}{|\widetilde{x}'|^{4}}
+\frac{2\sum_{k=1}^{N}\widetilde{x}'_{k}(\widetilde{X}_{k}f)(x)}{|\widetilde{x}'|^{4}} \\
& \quad -
(N-2)\frac{\widetilde{x}'\cdot(\nabla_{H}f)(x)}{|\widetilde{x}'|^{4}} \\
& =-
(N-3)\frac{\widetilde{x}'\cdot(\nabla_{H}f)(x)}{|\widetilde{x}'|^{4}}-
\frac{\sum_{j,k=1}^{N}\widetilde{x}'_{j}\widetilde{x}'_{k}(\widetilde{X}_{j}\widetilde{X}_{k}f)(x)}{|\widetilde{x}'|^{4}}.
\end{aligned}
\end{equation}

Now putting \eqref{J1}-\eqref{J4} in \eqref{heisen_TT0}, we obtain
\begin{equation}\label{heisen_TT1}
\begin{split}
(T^{+}_{\alpha,\beta}T_{\alpha,\beta}f)(x)=&(\L^{2}f)(x)-(2\beta-(N-4)\alpha)\frac{(\L f)(x)}{|\widetilde{x}'|^{2}}\\
&+(2(N-2)\alpha+4\beta-(N-3)\alpha^{2})\frac{\widetilde{x}'\cdot(\nabla_{H} f)(x)}{|\widetilde{x}'|^{4}}\\
&+(\beta^{2}+2(N-4)\beta-(N-4)\alpha\beta)\frac{f(x)}{|\widetilde{x}'|^{4}}\\
&-\alpha(\alpha-4)\frac{\sum_{j,k=1}^{N}\widetilde{x}'_{j}\widetilde{x}'_{k}(\widetilde{X}_{j}\widetilde{X}_{k}f)(x)}{|\widetilde{x}'|^{4}}\\
& + 2\alpha \sum_{j=1}^{n}\frac{({x}_{j}Y_j-y_j X_j)  Tf(x)}{|\widetilde{x}'|^{2}}.
\end{split}
\end{equation}
Now using the nonnegativity of $T_{\alpha,\beta}^{+}T_{\alpha,\beta}$ and integrating by parts, we get
$$\int_{\mathbb{H}_{n}}|(T_{\alpha,\beta}f)(x)|^{2}dx=\int_{\mathbb{H}_{n}}\overline{f(x)}(T^{+}_{\alpha,\beta}T_{\alpha,\beta}f)(x)dx\geq0,$$
wherefor brevity we can write $dx$ for the integral on the Heisenberg group $\mathbb{H}_{n}$.
Putting \eqref{heisen_TT1} into this inequality, one calculates
\begin{equation}\label{heisen_TT_int_part1}
\begin{split}
\int_{\mathbb{H}_{n}}|(\mathcal{L}f)(x)|^{2}dx-(2\beta-(N-4)\alpha)\int_{\mathbb{H}_{n}}\frac{\overline{f(x)}(\L f)(x)}{|\widetilde{x}'|^{2}}dx\\
+(2(N-2)\alpha+4\beta-(N-3)\alpha^{2})\int_{\mathbb{H}_{n}}\frac{\overline{f(x)}(\widetilde{x}'\cdot (\nabla_{H}f)(x))}{|\widetilde{x}'|^{4}}dx\\
+(\beta^{2}+2(N-4)\beta-(N-4)\alpha\beta)\int_{\mathbb{H}_{n}}\frac{|f(x)|^{2}}{|\widetilde{x}'|^{4}}dx\\
-\alpha(\alpha-4)\sum_{j,k=1}^{N}\int_{\mathbb{H}_{n}}\frac{\overline{f(x)}\widetilde{x}'_{j}\widetilde{x}'_{k}
(\widetilde{X}_{j}\widetilde{X}_{k}f)(x)}{|\widetilde{x}'|^{4}}dx\\
+ 2\alpha \sum_{j=1}^{n}\int_{\mathbb{H}_{n}}\frac{\overline{f(x)} ({x}_{j}Y_j-y_j X_j)  Tf(x)}{|\widetilde{x}'|^{2}}
dx\geq0.
\end{split}
\end{equation}
Let us analyse the last term in this inequality. Using formula \eqref{EQ:Heis-nots}, we have
\begin{equation}\label{EQ:Heis-tan1}
{x}_{j}Y_j-y_j X_j=x_j(\partial_{y_j}+\frac{x_j}{2}\partial_t)-y_j(\partial_{x_j}-\frac{y_j}{2}\partial_t)
=x_j\partial_{y_j}-y_j\partial_{x_j}+\frac{x_j^2+y_j^2}{2}\partial_t.
\end{equation}
Consequently, we have
\begin{equation}\label{EQ:Heis-tan2}
\sum_{j=1}^n ({x}_{j}Y_j-y_j X_j)=Z+\frac{|\widetilde{x}'|^2}{2}T,
\end{equation}
with the tangential derivative $Z$ defined in \eqref{EQ:Lap2} and $T=\partial_t$.

For the other terms in \eqref{heisen_TT_int_part1}, using \eqref{formula1}, we have
\begin{equation}\label{heisen_TT_int_part2}
\begin{aligned}
\int_{\mathbb{H}_{n}}\frac{\overline{f(x)}(\L f)(x)}{|\widetilde{x}'|^{2}}dx
& =-\sum_{j=1}^{N}\int_{\mathbb{H}_{n}}\frac{\overline{(\widetilde{X}_{j}f)(x)}(\widetilde{X}_{j}f)(x)}
{|\widetilde{x}'|^{2}} \\
& \quad -\sum_{j=1}^{N}\int_{\mathbb{H}_{n}}\overline{f(x)}(-2)|\widetilde{x}'|^{-3}\widetilde{X}_{j}
|\widetilde{x}'|(\widetilde{X}_{j}f)(x)dx \\
& =2\int_{\mathbb{H}_{n}}\frac{\overline{f(x)}(\widetilde{x}'\cdot (\nabla_{H}f)(x))}{|\widetilde{x}'|^{4}}dx
-\int_{\mathbb{H}_{n}}\frac{|(\nabla_{H}f)(x)|^{2}}{|\widetilde{x}'|^{2}}dx
\end{aligned}
\end{equation}
and
$$
\sum_{j,k=1}^{N}\int_{\mathbb{H}_{n}}\frac{\overline{f(x)}\widetilde{x}'_{j}\widetilde{x}'_{k}(\widetilde{X}_{j}\widetilde{X}_{k}f)(x)}{|\widetilde{x}'|^{4}}dx$$
$$=-(N-1)\sum_{k=1}^{N}\int_{\mathbb{H}_{n}}\frac{\overline{f(x)}\widetilde{x}'_{k}(\widetilde{X}_{k}f)(x)}{|\widetilde{x}'|^{4}}dx
-2\sum_{k=1}^{N}\int_{\mathbb{H}_{n}}\frac{\overline{f(x)}\widetilde{x}'_{k}(\widetilde{X}_{k}f)(x)}{|\widetilde{x}'|^{4}}dx$$
$$-\sum_{j,k=1}^{N}\int_{\mathbb{H}_{n}}\frac{\widetilde{x}'_{j}\widetilde{x}'_{k}\overline{(\widetilde{X}_{j}f)(x)}
(\widetilde{X}_{k}f)(x)}{|\widetilde{x}'|^{4}}dx
-\sum_{j,k=1}^{N}\int_{\mathbb{H}_{n}}\overline{f(x)}\widetilde{x}'_{j}\widetilde{x}'_{k}(\widetilde{X}_{k}f)(x)(-4)|\widetilde{x}'|^{-5}\widetilde{X}_{j}|\widetilde{x}'|dx$$$$=
-(N+1)\sum_{k=1}^{N}\int_{\mathbb{H}_{n}}\frac{\overline{f(x)}\widetilde{x}'_{k}(\widetilde{X}_{k}f)(x)}{|\widetilde{x}'|^{4}}dx
+4\sum_{j,k=1}^{N}\int_{\mathbb{H}_{n}}\frac{\overline{f(x)}(\widetilde{x}'_{j})^{2}\widetilde{x}'_{k}(\widetilde{X}_{k}f)(x)}
{|\widetilde{x}'|^{6}}dx$$$$
-\int_{\mathbb{H}_{n}}\frac{|\widetilde{x}'\cdot (\nabla_{H}f)(x)|^{2}}{|\widetilde{x}'|^{4}}dx
=-(N-3)\int_{\mathbb{H}_{n}}\frac{\overline{f(x)}(\widetilde{x}'\cdot (\nabla_{H}f)(x))}{|\widetilde{x}'|^{4}}dx$$
\begin{equation}\label{heisen_TT_int_part3}
-\int_{\mathbb{H}_{n}}\frac{|\widetilde{x}'\cdot (\nabla_{H}f)(x)|^{2}}{|\widetilde{x}'|^{4}}dx.
\end{equation}
Putting \eqref{EQ:Heis-tan2},  \eqref{heisen_TT_int_part2} and \eqref{heisen_TT_int_part3} in \eqref{heisen_TT_int_part1}, we obtain
$$
\int_{\mathbb{H}_{n}}|(\mathcal{L}f)(x)|^{2}dx
+(2\alpha(N-4)-4\beta+\alpha(\alpha-4)(N-3)-\alpha^{2}(N-3)+2\alpha(N-2)+4\beta)$$
$$\times\int_{\mathbb{H}_{n}}\frac{\overline{f(x)}(\widetilde{x}'\cdot (\nabla_{H}f)(x))}{|\widetilde{x}'|^{4}}dx+(\beta^{2}+2(N-4)\beta-(N-4)\alpha\beta)\int_{\mathbb{H}_{n}}\frac{|f(x)|^{2}}{|\widetilde{x}'|^{4}}dx$$
$$-(\alpha(N-4)-2\beta)\int_{\mathbb{H}_{n}}\frac{|(\nabla_{H} f)(x)|^{2}}{|\widetilde{x}'|^{2}}dx$$$$
+\alpha(\alpha-4)\int_{\mathbb{H}_{n}}\frac{|\widetilde{x}'\cdot (\nabla_{H}f)(x)|^{2}}{|\widetilde{x}'|^{4}}dx$$
$$+ 2\alpha  \int_{\mathbb{H}_{n}} \frac{\overline{f(x)} (Z+\frac{|\widetilde{x}'|^2}{2}T)  Tf(x)}{|\widetilde{x}'|^{2}}
dx\geq0,
$$
which implies \eqref{heisen_Rellich1}.  Then, we can also obtain \eqref{heisen_Rellich2} from \eqref{heisen_Rellich1} using the Cauchy-Schwarz inequality \eqref{Cauchy-Schwarz} for $\alpha(\alpha-4)\geq0$ in the last estimate.

Thus, we have obtained \eqref{heisen_Rellich1} and \eqref{heisen_Rellich2} using the nonnegativity of $T_{\alpha,\beta}^{+}T_{\alpha,\beta}$. Now let us show \eqref{heisen_Rellich11} and \eqref{heisen_Rellich22} using the nonnegativity of $T_{\alpha,\beta}T_{\alpha,\beta}^{+}$. By writing \eqref{strat_TT1_ab} on Heisenberg group and subtracting the expression $T_{\alpha,\beta}^{+}T_{\alpha,\beta}$ in \eqref{heisen_TT1} from this, we get
\begin{equation}\label{heisen_TT11}
\begin{split}
& (T_{\alpha,\beta}T^{+}_{\alpha,\beta}f)(x) \\
&=(\L^{2}f)(x)-(2\beta-N\alpha)\frac{(\L f)(x)}{|\widetilde{x}'|^{2}}\\
&+(6(2-N)\alpha+4\beta-(N-3)\alpha^{2})\frac{\widetilde{x}'\cdot(\nabla_{H} f)(x)}{|\widetilde{x}'|^{4}}\\
&+(\beta^{2}+2(N-4)\beta-N\alpha\beta+2\alpha^{2}(N-2)-2\alpha(N-4)(N-2))\frac{f(x)}{|\widetilde{x}'|^{4}}\\
&-\alpha(\alpha+4)\frac{\sum_{j,k=1}^{N}\widetilde{x}'_{j}\widetilde{x}'_{k}(\widetilde{X}_{j}\widetilde{X}_{k}f)(x)}{|\widetilde{x}'|^{4}}\\
& - 2\alpha \sum_{j=1}^{n}\frac{({x}_{j}Y_j-y_j X_j)  Tf(x)}{|\widetilde{x}'|^{2}}.
\end{split}
\end{equation}
Then using the nonnegativity of $T_{\alpha,\beta}T^{+}_{\alpha,\beta}$ and integrating by parts, one has
$$\int_{\mathbb{H}_{n}}|(T_{\alpha,\beta}f)(x)|^{2}dx=\int_{\mathbb{H}_{n}}\overline{f(x)}(T_{\alpha,\beta}T^{+}_{\alpha,\beta}f)(x)dx\geq0.$$
Putting \eqref{heisen_TT11} into this inequality, one gets
\begin{equation}\label{heisen_TT_int_part11}
\begin{split}
& \int_{\mathbb{H}_{n}}|(\mathcal{L}f)(x)|^{2}dx-(2\beta-N\alpha)\int_{\mathbb{H}_{n}}\frac{\overline{f(x)}(\L f)(x)}{|\widetilde{x}'|^{2}}dx\\&
+(6(2-N)\alpha+4\beta-(N-3)\alpha^{2})\int_{\mathbb{H}_{n}}\frac{\overline{f(x)}(\widetilde{x}'\cdot (\nabla_{H}f)(x))}{|\widetilde{x}'|^{4}}dx\\&
+(\beta^{2}+2(N-4)\beta-N\alpha\beta+2\alpha^{2}(N-2)-2\alpha(N-4)(N-2))\int_{\mathbb{H}_{n}}
\frac{|f(x)|^{2}}{|\widetilde{x}'|^{4}}dx\\&
-\alpha(\alpha+4)\sum_{j,k=1}^{N}\int_{\mathbb{H}_{n}}\frac{\overline{f(x)}\widetilde{x}'_{j}\widetilde{x}'_{k}
(\widetilde{X}_{j}\widetilde{X}_{k}f)(x)}{|\widetilde{x}'|^{4}}dx\\&
- 2\alpha \sum_{j=1}^{n}\int_{\mathbb{H}_{n}}\frac{\overline{f(x)} ({x}_{j}Y_j-y_j X_j)  Tf(x)}{|\widetilde{x}'|^{2}}
dx\geq0.
\end{split}
\end{equation}
Then, similarly as for $T^{+}_{\alpha,\beta}T_{\alpha,\beta}$, i.e. putting \eqref{EQ:Heis-tan2},  \eqref{heisen_TT_int_part2} and \eqref{heisen_TT_int_part3} in \eqref{heisen_TT_int_part11}, we obtain \eqref{heisen_Rellich11} and \eqref{heisen_Rellich22}.
\end{proof}

\end{document}